\documentclass[a4paper,11pt, reqno]{amsart}

\usepackage{amsfonts}
\usepackage{amssymb}
\usepackage{amsmath}
\usepackage{graphicx}
\usepackage{wrapfig}
\usepackage{amsthm}
\usepackage{amsrefs}
\usepackage{latexsym}
\usepackage{color}

\theoremstyle{plain}

\newcommand{\pt}{\partial}

\newcommand{\re}{\mathbb R}
\newcommand{\R}{\mathbb R}

\def\<{\langle }

\renewcommand{\>}{\rangle }

\newcounter{thm}[section]

\newtheorem{theo}[thm]{Theorem}

\newtheorem{lem}[thm]{Lemma}

\theoremstyle{definition}

\newtheorem{rem}{Remark}[section]
\numberwithin{equation}{section}

\begin{document}
%旧タイトル
%\title[System of cubic NLS equations in 1$d$]{Asymptotic behvavior in time of solution to system of cubic nonlinear
%Schr\"{o}dinger equations in one space dimension} 
\title[System of cubic NLS equations in 1$d$]{Polynomial deceleration for a system of cubic nonlinear 
Schr\"{o}dinger equations in one space dimension} 
\author[N. Kita]{Naoyasu Kita} 
\address{Faculty of Advanced Science and Technology, Kumamoto University, 
Kumamoto, 860-8555, Japan}
\email{nkita@kumamoto-u.ac.jp}
\author[S. Masaki]{Satoshi Masaki}
\address{Department of systems innovation, Graduate school of Engineering Science, Osaka University,
Toyonaka Osaka, 560-8531, Japan}
\email{masaki@sigmath.es.osaka-u.ac.jp} 
\author[J. Segata]{Jun-ichi Segata}
\address{Faculty of Mathematics, Kyushu University, Fukuoka, 819-0395, Japan} 
\email{segata@math.kyushu-u.ac.jp}
\author[K. Uriya]{Kota Uriya}
\address{Department of Applied Mathematics, Faculty of Science, Okayama University of Science, Okayama,
700-0005, Japan}
\email{uriya@xmath.ous.ac.jp}
%\date{}
\maketitle

\begin{abstract}
In this paper, we consider the initial value problem 
of a specific system of cubic nonlinear Schr\"{o}dinger equations. 
Our aim of this research is 
to specify the asymptotic profile of the solution in $L^{\infty}$ as $t \to \infty$. 
It is then revealed that the solution decays slower than a linear solution does.
Further, the difference of the decay rate is a polynomial order.
This deceleration of the decay is due to an amplification effect by the nonlinearity.
This nonlinear amplification phenomena was previously known for several specific systems, however the deceleration of the decay in these results was
by a logarithmic order.
As far as we know, the system studied in this paper is the first model in that the deceleration
in a polynomial order is justified.
%Our result is obtained under certain condition on the coefficients 
%of the nonlinearity, which was not assumed in the original research 
%of \cite{MSU1}.
\end{abstract}

%\markboth{\sc N.KITA, S.MASAKI, J.SEGATA AND K.URIYA}{SYSTEM OF CUBIC NLS EUATION IN 1D}
%
%\centerline {\large{\bf ASYMPTOTIC BEHAVIOR IN TIME OF SOLUTION TO SYSTEM}}
%\centerline{\large{\bf OF CUBIC NONLINEAR SCHR\"{O}DINGER EQUATIONS}}
%\centerline{\large{\bf -- CASE OF POLYNOMIALLY GROWING $\mathcal{F}U(-t)u_2$}} 
%%\centerline{\color{red}{(draft)}\color{black}}
%
%\bigskip
%
%\begin{center}
%\centerline{\large NAOYASU KITA, SATOSHI MASAKI, JUN-ICHI SEGATA }
%\centerline{\large AND  KOTA URIYA}
%\end{center}
%
%\maketitle

%
%\footnote[0]{2010 {\it Mathematics Subject Classification.} 
%Primary 35Q55; Secondary 35A22, 35B40.}
%\footnote[0]{{\it Key Words and Phrases.} 
%Nonlinear Schr\"{o}dinger equation, Asymptotic behavior of solutions, 
%Long range scattering.}

%\bigskip
%
%\noindent \textsc{Abstract}. In this paper, we consider the initial value problem 
%of a system of nonlinear Schr\"{o}dinger equations. Our aim of this research is 
%to specify the asymptotic leading terms of the solution in $L^{\infty}$ as $t \to \infty$. 
%We find that the solution possesses polynomially weakly decaying and growing 
%profiles. Our result is obtained under certain condition on the coefficients 
%of the nonlinearity, which was not assumed in the original research 
%of \cite{MSU1}.
% 
%\bigskip

\section{Introduction}
\subsection{System of cubic NLS equations in one space dimension}
In this paper, we consider the large time behavior of solutions to 
%We consider 
the Cauchy problem of  the following system of cubic nonlinear Schr\"{o}dinger (NLS) equations in one space dimension:
\begin{equation}\label{NLS}
\left\{ 
\begin{aligned}
   & \mathcal{L} u_1  
       = 3 \lambda_1 |u_1|^2 u_1, \quad && t \in \re, x \in \re,\\
   & \mathcal{L} u_2 
       = \lambda_6 (2|u_1|^2 u_2 + u^2_1 \overline{u}_2), \quad && t \in \re, x \in \re,\\
    &u_1(0, x) = u_{1, 0} (x), \quad u_2 (0, x) = u_{2,0} && x \in \re,
\end{aligned}
\right. 
\end{equation}
where $(u_1,u_2) : \re \times \re \to \mathbb{C}^2$ ($j=1,2$) is a pair of unknown functions, 
${{\mathcal L}}= i \partial_t+(1/2)\partial^2_x$, and $\lambda_1$ and $\lambda_6$ are real constants 
satisfying $(\lambda_1 , \lambda_6) \neq (0,0)$ and 
\begin{equation}\label{E:cond}
	(\lambda_6 - \lambda_1) (\lambda_6 - 3 \lambda_1) < 0. 
\end{equation}
We give a precise assumption on the data $(u_{1,0},u_{2,0})$ later.
Throughout the paper, we use the notation $\eta = 3\lambda_1 - 2\lambda_6$ and $\mu = \sqrt{\lambda_6^2 - \eta^2}$.
We remark that \eqref{E:cond} reads as $\lambda_6^2 - \eta^2 >0$, and hence it implies
 that $\mu$ is a real number.
The system \eqref{NLS} is a special case of
\begin{equation}\label{NLSgen}
\left\{ 
\begin{aligned}
   & \mathcal{L} u_1  
       = 3 \lambda_1 |u_1|^2 u_1+ \lambda_2 (2|u_1|^2 u_2 + u^2_1 \overline{u}_2)\\
       	&\qquad\qquad+ \lambda_3 (2u_1|u_2|^2  + \overline{u}_1u^2_2 )+  3 \lambda_4 |u_2|^2 u_2, \\
   & \mathcal{L} u_2 
       = 3 \lambda_5 |u_1|^2 u_1+ \lambda_6 (2|u_1|^2 u_2 + u^2_1 \overline{u}_2)\\
       &\qquad\qquad + \lambda_7 (2u_1|u_2|^2  + \overline{u}_1u^2_2 )+  3 \lambda_8 |u_2|^2 u_2.
\end{aligned}
\right. 
\end{equation}
Classification of systems of this form is introduced in \cites{MSU1,MSU2,Mpre}.

It is widely known that the cubic nonlinearity is critical in one space dimension in view of the large time behavior 
of solutions to several typical dispersive equations such as Schr\"odinger equation and Klein-Gordon equation.
More precisely, the nonlinear effect can be seen in the large time behavior of solutions.
The cubic nonlinearity is critical also for systems of dispersive equations.
There is a variety of types on the large time behavior of solutions to systems, even if we restrict ourselves 
to the systems of the form \eqref{NLSgen}.

To compare with, let us first recall that the asymptotic behavior of 
a solution to the linear Schr\"odinger equation $\mathcal{L} u=0$ is given as
\[
	%u(t) = 
	t^{-\frac12} e^{i\frac{x^2}{2t}-i \frac{\pi}4}  \mathcal{F} [u(0)](\tfrac{x}t). %+ O(t^{-\frac54})
\]
%in $L^\infty$ as $t\to \infty$ if $\mathcal{F} u(0) \in H^2$.
In particular, the decay order of the linear solution in the $L^\infty$-topology is $t^{-1/2}$.

Ozawa \cite{Ozawa} showed that the asymptotic profile of
a small good solution to the cubic equation 
\begin{equation}\label{cNLS}
	\mathcal{L} u=\lambda |u|^2 u, \quad \lambda \in \R
\end{equation}
is given as
\begin{equation}
	t^{-\frac12} e^{i\frac{x^2}{2t}-i \frac{\pi}4}  (\widehat{u_+} e^{-i \lambda |\widehat{u_+}|^2 \log t })(\tfrac{x}t),
\end{equation}
where $u_+$ is a suitable function (See also \cite{GO,HN}).
Notice that, due to the presence of the gauge-invariant cubic nonlinear term,
\emph{a logarithmic phase correction} is involved in the asymptotic profile
(See \cite{KN,NST,Uriya} for a similar result for other models).
The first equation of the system \eqref{NLS} is nothing but the single cubic equation and hence
its asymptotic behavior is given above.

It is also known that the cubic nonlinearity may change the decay of a solution.
One such example is known as the so-called \emph{nonlinear dissipation} phenomenon.
The typical case is \eqref{cNLS} with a complex coefficient $\lambda \in \mathbb{C} \setminus \mathbb{R}$.
When $\Im \lambda>0$ then \eqref{cNLS} is dissipative for the positive time direction,
and it has shown that the small solution decays faster than 
the free Schr\"odinger evolution (see \cite{HLN3,HLN4,Hoshino,Kita2,OS,OU,Sato,Shimomura}).
Note the logarithmic phase correction do not change the decay of solution.

Oppositely, there is a models in which the decay of a solution is decelerated due to the nonlinear effect.
We call this phenomenon as the \emph{(weak) nonlinear amplification}.
A typical example is
\eqref{NLS} with the choice $\lambda_6=\lambda_1 (\neq 0)$ or $\lambda_6 = 3 \lambda_1 (\neq0)$.
The second, third, and fourth authors show in
 \cite{MSU1} that \eqref{NLS} with these two choices admits a solution of which asymptotic profile
is the pair of
\[
		t^{-\frac12} e^{i\frac{x^2}{2t}-i \frac{\pi}4}  (W_1 e^{-3i \lambda_1 |W_1|^2 \log t })(\tfrac{x}t),
\]
and
\[
	t^{-\frac12} e^{i\frac{x^2}{2t}-i \frac{\pi}4}  ((W_2 + W \log t) e^{-3i \lambda_1 |W_1|^2 \log t })(\tfrac{x}t)
%t^{-\frac12}\left\{W\left(\frac{x}{t}\right)\log t+W_2\left(\frac{x}{t}\right)\right\}
%e^{\frac{ix^2}{2t}-i3\lambda_1\left|W_1\left(\frac{x}{t}\right)\right|^2\log t-i\frac{\pi}{4}},
\]
by a suitable pair $(W_1,W_2)$ of functions,
where $W$ is given by 
\begin{align*} 
W=
\left\{
\begin{aligned}
&2\lambda_1W_1\Im\left[W_1\overline{W_2}\right]      &\text{if}\ \lambda_6=\lambda_1, \\
&-6i\lambda_1W_1\Re\left[W_1\overline{W_2}\right] &\ \text{if}\ \lambda_6=3\lambda_1.
\end{aligned}
\right.
\end{align*} 
Besides the logarithmic phase correction,
a \emph{logarithmic amplitude correction}, denoted by $W$, shows up in 
the asymptotic profile of the second component.
One has
\[
	\|u_2(t)\|_{L^\infty} \sim t^{-\frac12}(1+ \|W\|_{L^\infty} \log t) 
\]
as $t\to\infty$.
We remark that the amplitude correction, which makes the decay of the solution slower,
is due to the nonlinear effect between two components of unknowns.
Indeed, $W_1=0$ implies $W=0$, which tells us that, without the presence of the first
component of the solution, the amplitude correction does not occur.
For a generic non-trivial solution, we have $W\neq0$.

As for \eqref{NLS},
the case $(\lambda_6-\lambda_1)(\lambda_6-3\lambda_1)=0$ is a threshold.
Indeed, it is proved in \cite{MSU1} that 
the nonlinear amplification %, deceleration of the decay
do not take place in the case $(\lambda_6-\lambda_1)(\lambda_6-3\lambda_1)>0$.
Further, it is revealed that
 the asymptotic behavior of $u_2$ in this case is a sum of two parts which oscillate in a different way.
The oscillations of the two parts get close each other as $(\lambda_6-\lambda_1)(\lambda_6-3\lambda_1)\downarrow0$
and in the limit case, that is, at the threshold case, they become the same. 
As mentioned below, the asymptotic profile is given as a solution to a corresponding ODE system.
At least in this ODE level,
we see that  the logarithmic amplitude correction
is produced by the coincidence of the oscillations (cf. resonance phenomenon in the theory of ODEs).

Another example of a system in which the nonlinear amplification occurs is
\begin{equation}\label{NLSpre}
\left\{ 
\begin{aligned}
   & \mathcal{L} u_1  
       = 0, \quad && t \in \re, x \in \re,\\
   & \mathcal{L} u_2 
       =  3|u_1|^2 u_1 ,\quad && t \in \re, x \in \re.
\end{aligned}
\right. 
\end{equation}
This system is  also studied in \cite{MSU1} and a similar slowly-decaying solution is found.
However, an ODE analysis shows that
the mechanism of the appearance of the logarithmic amplitude correction is slightly different.
We remark that a similar result is previously known for the corresponding system of Klein-Gordon equations, i.e., \eqref{NLSpre} with $\mathcal{L}=\pt_t^2- \pt_x^2 + 1$ (see \cite{Sunagawa}).

In this paper, we study the other side of the threshold, that is,
the case \eqref{E:cond}, of the system \eqref{NLS}.
It will turn out that there is a stronger deceleration effect.
Namely, the asymptotic profile of the second component involves a \emph{polynomial amplitude correction}.
As far as we know, the deceleration of the time-decay by a polynomial order is not known before at least for 
the Schr\"odinger equations/systems and Klein-Gordon equation/systems.

One obstacle in justifying the existence of such a slowly-decaying solution
comes from the fact that the criticalness of the cubic nonlinearity in one space dimension
is related to the critical decay rate $O(t^{-1/2})$ in $L^\infty(\R)$.
To explain this respect in more detail,
let us consider the power type nonlinear Schr\"odinger equation $\mathcal{L} u = |u|^{p-1} u$, where $p>1$.
By means of the propagator $U(t) = e^{\frac{i}2 t \pt_x^2}$, it is rewritten as
\[
	i\pt_t (U(-t)u(t)) = U(-t) (|u|^{p-1} u).
\]
By the conservation of mass, under the assumption that $u$ decays like $O(t^{-1/2})$ in $L^\infty(\R)$ as $t\to\infty$,
the right hand side becomes an integrable $L^2$-valued function of time on $\R_+$ if $p>3$.
Note that the integrability means that the solution $u(t)$ scatters (See \cite{TY}).
The only-if part, i.e., the failure of scattering for $p\le3$ is well-known (See \cite{Barab,Strauss}).
This is how the cubic nonlinearity appears as a critical index.
The argument also suggests that the cubic nonlinearity is not the critical index but a supercritical one when solutions decay at a slower rate than $O(t^{-1/2})$.
From a technical point of view, the slower decay makes it difficult to obtain a closed estimate.

As for \eqref{NLS}, the special structure of the system enables us to obtain a closed estimate in spite of the presence of the deceleration effect.
Firstly, the first equation of the system is nothing but the single cubic equation \eqref{cNLS}.
Hence, by applying the previous result on the single equation, we obtain the asymptotic behavior of
the first component $u_1$.
We remark that $u_1$ decays at the critical rate, the same rate as the linear solution does.
Secondly, the equation for $u_2$ is linear with respect to $u_2$.
Hence, the critical decay of $u_1$ is sufficient to obtain a global bound on $u_2$.
 
In many case, the asymptotic profile of a solution is given by the corresponding ODE equation/system.
Roughly speaking, the asymptotic profile of a good solution to a system of cubic Schr\"odinger equation
\[
	\mathcal{L} u_\ell = N_\ell(u_1,u_2,\dots,u_L), \quad (\ell=1,2,\dots,L),
\]
where $N_\ell$ is a homogeneous cubic nonlinearities,
is described by a solution $(A_1,A_2,\dots,A_L)=(A_1(\tau;\xi),A_2(\tau;\xi),\dots,A_L(\tau;\xi))$ parametrized by $\xi \in \R$ to
the system of ordinary equations
\begin{equation}\label{ODEgen}
	i A_\ell' = t^{-1} N_\ell(A_1,A_2,\dots,A_L) \quad (\ell=1,2,\dots,L).
\end{equation}
Namely, in some sense,
\[
		u_\ell (t,x) \sim t^{-\frac12} e^{i\frac{x^2}{2t}-i \frac{\pi}4}  A_\ell( t; \tfrac{x}t) \quad (\ell=1,2,\dots,L)
\]
as $t\to\infty$.
As seen below, our main result ensures that this is true for \eqref{NLS} with \eqref{E:cond} (See Remark \ref{R:ODE}).
By the previous results, we see that
this is also true for all systems introduced above, i.e., for the rest case of \eqref{NLS}, and \eqref{cNLS} and \eqref{NLSpre}.
Katayama and Sakoda \cite{KaSa} show that this is true for a class of systems of cubic Schr\"odinger equations.

There are two known systems for which the analysis of the corresponding ODE system suggests
that a deceleration of the decay by a polynomial order takes place.
One is
\begin{equation}\label{E:newODE1}
\left\{
\begin{aligned}
	&\mathcal{L} u_1=
	3 |u_1|^2u_1-3(2u_1|u_2|^2 +\overline{u_1}u_2^2),\\
	&\mathcal{L} u_2=
	3 (2|u_1|^2u_2+u_1^2\overline{u_2})-3|u_2|^2u_2,
\end{aligned}
\right.
\end{equation}
and the other is
\begin{equation}\label{E:newODE2}
\left\{
\begin{aligned}
	&\mathcal{L} u_1=
	3 |u_1|^2u_1-(2u_1|u_2|^2 +\overline{u_1}u_2^2),\\
	&\mathcal{L} u_2=
	(2|u_1|^2u_2+u_1^2\overline{u_2})-3|u_2|^2u_2.
\end{aligned}
\right.
\end{equation}
The corresponding ODE systems are studied in \cite{MSU2} and it is shown that 
the ODE systems admit solutions with polynomial deceleration. 
However, no PDE result is available on the large time behavior of solutions to these two systems.
It is needless to say that these systems  do not possess a good structure as \eqref{NLS} does.

Finally, let us briefly mention the \emph{strong nonlinear amplification} phenomenon.
In the case \eqref{cNLS} with a complex coefficient $\lambda \in \mathbb{C} \setminus \mathbb{R}$,
the  nonlinear dissipation occurs either forward or backward in time, as mentioned above.
For this system, a different type of nonlinear amplification phenomenon takes place in the opposite time direction.
The amplification is stronger than what we consider in this paper:
At the ODE level, any nontrivial solution to
the corresponding ODE equation $i A' = \lambda |A|^2 A$ ($\Im \lambda > 0$) blows up in finite time for positive time direction.
At the PDE level, the first author \cite{Kita1} show that there exists an arbitrarily small (in $L^2$) initial datum which gives a solution 
blowing up in finite time, if $\Im \lambda > 0$ is sufficiently large compared with $|\Re \lambda|$. 
%In the case where the smallness assumption 
%on the initial data is not considered, 
The existence of blow-up solutions is studied also in \cite{CMZhao, CMZheng, KM} without a smallness assumption on the data. 
It is also known that there exists a system of the form \eqref{NLSgen} (with real coefficients) for which strong nonlinear amplification phenomenon of this type occurs (see \cite{MSU1}).

\subsection{Main result}
To state our main result, let us introduce several notations.
For $a \in \re$, $\langle a \rangle := \sqrt{1+|a|^2}$.
We denote $A \lesssim B$ if there exists a constant $C > 0$ 
such that $A \le C B$ holds and $ A \sim B $ if $ A \lesssim B \lesssim A $.
It is also used with subscripts such as $A \lesssim_{a,b} B$ when we
emphasize that the hidden constant $C$ depends on other quantities $a$ and $b$.
For $1 \le q \le \infty$,
let $L^q = L^q ( \re) $ denote the usual Lebesgue space
with the norm 
\begin{equation*}
\left\| f \right\| _{L^q }
  =
  \left\{
  \begin{aligned}
 & \left( \int_{\re} \left| f \left( x\right) \right| ^{q}dx\right) ^{\frac{1}{q}}
 && 1\leq q<\infty,  \\
 & \underset{x\in \re}{\text{ess.sup}}\left| f \left( x\right) \right| 
 && q=\infty.
 \end{aligned}
 \right.
\end{equation*}
$\mathcal{F}$ and $\mathcal{F}^{-1}$ stand for
the Fourier and inverse Fourier transforms, respectively:   
\begin{align*}
&
  (\mathcal{F}f)(\xi )%=\hat{f} 
    := (2\pi)^{-\frac12} 
        \int_{\re} e^{-ix\cdot \xi} f(x) \ dx, &
& (\mathcal{F}^{-1}f)(x)
     := (2\pi)^{-\frac12} \int_{\re} e^{ix\cdot \xi}f(\xi) \ d \xi.
\end{align*}
The weighted Sobolev space $H^{s,m}$ is defined by 
\begin{align*}
H^{s,m} &= \{ f \in \mathcal{S}' \ ; \|f\|_{H^{s,m}} <\infty\},\\
\|f\|_{H^{s,m}}&=\|\langle x \rangle^m 
                       \mathcal{F}^{-1} \langle \xi \rangle^s{{\mathcal F}}f\|_{L_{x}^{2}}, 
\end{align*}
where $ s, m \in \re $. 
For $t\neq0$,
we let $M(t)=e^{\frac{i}{2t} |x| ^{2}}\times$ be a multiplication operator.
We define the dilation operator by 
\begin{equation*}
   \left( D(t)f\right) \left( x\right) 
        =t^{-1/2}f (x/t)e^{-i\frac{\pi}{4}}
\end{equation*}
for $t\neq 0$. It is well-known that the Schr\"{o}dinger 
group $U(t) = \exp (it \partial^2_x /2)$ can be decomposed as 
\begin{equation*}
U\left( t\right) =M(t) D(t) \mathcal{F}M(t).
\end{equation*}
Since $U(-t) = U(t)^{-1}$, we also have 
%\begin{equation*}
$U\left( -t\right) =M(t)^{-1}\mathcal{F}^{-1}D(t)^{-1}M(t)^{-1}$.
%\end{equation*}
The standard generator of  the Galilean transformation is given as 
\begin{equation*}
J(t) =U(t) xU( -t) =x+it\partial_x.
\end{equation*}
Note that the operators $J(t)$ and $\mathcal{L}$
%$i\partial_t +(1/2)\partial^2_x$ 
commute. 
We often simply denote $J$ when the variable $t$ is clear from the context.

Recall that $\eta=3\lambda_1 - 2\lambda_6$ and $\mu = \sqrt{\lambda_6^2-\eta^2}$.
Note that $|\mu\pm i \eta| = |\lambda_6|>0$ holds. 
%Note that \eqref{E:cond} reads  as $\lambda_6^2 - \eta^2 >0$ and hence as $\mu \in \re$.
Our main result is the following.

\begin{theo}\label{1.1}
%Suppose $(\lambda_6 - \lambda_1)(\lambda_6 - 3\lambda_1) < 0$.  
Let $\varepsilon_j := \|u_{j,0}\|_{H^1} + \|u_{j,0}\|_{H^{0,1}} \  
(j=1,2)$. % and let $0<\gamma_1 < \gamma_2 < 1/100$. 
Then there exists $\varepsilon_0 > 0$ % and $K>0$ 
such that for any $u_{j,0} \in H^1 \cap H^{0,1}$ with  
$\varepsilon_1 \le \varepsilon_0$ and without size-restriction on $\varepsilon_2$, 
there exists a unique global solution 
$(u_1,u_2) \in C(\re ; H^1 (\re) \cap H^{0,1}(\re))^2$ to $(\ref{NLS})$ satisfying  
\begin{equation}\label{upper_bound}
\left.
\begin{aligned}
&\|u_1(t)\|_{H^1}+\|Ju_1(t)\|_{L^2} \lesssim 
\varepsilon_1 \langle t \rangle^{C\varepsilon_1^{2}}, \quad 
&& \|u_1(t)\|_{L^{\infty}} \lesssim \varepsilon_1 \langle t \rangle^{-1/2}, \\
&\|u_2(t)\|_{H^1}+\|Ju_2(t)\|_{L^2} 
\lesssim \varepsilon_2 \langle t \rangle^{C\varepsilon_1^{2}}, \quad 
&& \|u_2(t)\|_{L^{\infty}} \lesssim \varepsilon_2 \langle t \rangle^{-1/2 + C\varepsilon_1^2}.
\end{aligned}
\right.
\end{equation}
for any $t \in \re$.
%, where $J=x+it\partial_{x}$ and $\langle t\rangle=\sqrt{1+|t|^{2}}$.. 
Furthermore, there exist two functions $W_1, W_{2} \in L^{\infty}$ such that 
\begin{align}
%& u_1 (t) = t^{-\frac{1}{2}} W_1 (\tfrac{x}{t}) 
%                   e^{\frac{ix^2}{2t} - i3\lambda_1 |W_1(\frac{x}{t})|^2 \log t - i\frac{\pi}{4}}
%                   + O(t^{-\frac{3}{4} + \gamma_1}), \label{asymp1}\\
%& u_2 (t) = t^{-\frac{1}{2}} \tilde{W}_2 (t,\tfrac{x}{t}) 
%                   e^{\frac{ix^2}{2t} - i3\lambda_1 |W_1(\frac{x}{t})|^2 \log t - i\frac{\pi}{4}}
%                   + O(t^{-\frac{3}{4} + \gamma_1}) \label{asymp2}
& u_1 (t) = M(t) D(t) F_1(t)
                   + O(t^{-\frac{3}{4} + C \varepsilon_1^2}), \label{asymp1}\\
& u_2 (t) = M(t) D(t) F_2(t)
                   + O(t^{-\frac{3}{4} + C \varepsilon_1^2}) \label{asymp2}
\end{align}
in $L^{\infty}$ as $t \to \infty$, where
\begin{equation}
F_{1}(t,\xi)=
W_1(\xi)e^{-3i\lambda_{1}|W_1(\xi)|^{2}\log t},
\label{F11}
\end{equation}
\begin{equation}\label{F12}
F_{2}(t,\xi)
= \tilde{W}_2(t,\xi)e^{-3i\lambda_{1}|W_1(\xi)|^{2}\log t}
\end{equation}
and $\widetilde{W}_2(t, \xi)$ is defined by
$\widetilde{W}_2(t, \xi) :=W_2(\xi)$ if $W_1(\xi) = 0$ and
\begin{align*}
	\widetilde{W}_2(t, \xi) 
	&:=
	\frac1{2\mu} 
           \left(( \mu-i\eta) W_2(\xi) 
                    +i \lambda_6 \tfrac{W_1(\xi)^2}{|W_1(\xi)|^2}\overline{W_2 (\xi)} \right)
                 t^{-\mu |W_1(\xi)|^2} \\
 &\quad+ \frac1{2\mu}  
             \left((  \mu+i\eta) W_2 (\xi) 
                      -i \lambda_6 \tfrac{W_1(\xi)^2}{|W_1(\xi)|^2} \overline{W_2 (\xi)}\right)
                     t^{\mu |W_1(\xi)|^2}
\end{align*}
if $W_1(\xi) \neq 0$.
\end{theo}

\begin{rem}
If $W_1$ and $W_2$ are continuous then,
for each fixed $t > 0$, $\tilde{W}_2(t,\xi)$ is continuous in $\xi \in \re$
including the zero points of $W_1$. Indeed, one sees this from
$\frac{W_1^2}{|W_1|^2}(t^{\mu |W_1|^2} - t^{-\mu |W_1|^2}) = W_1^2(2\mu \log t + o(1))\to 0$
as $|W_1| \to 0$.
\end{rem}

\begin{rem}\label{R:ODE}
The pair $(F_1,F_2)$ given by \eqref{F11} and \eqref{F12}
is a unique solution to the following initial value problem of the system of ordinary differential equations:
\begin{equation}\label{ODEsys}
\left\{ 
\begin{aligned}
   & iF_1'  
       = 3 \lambda_1t^{-1}  |F_1|^2 F_1, \quad && t >0 \\
   & iF_2' 
       = \lambda_6 t^{-1} (2|F_1|^2 F_2 + F^2_1 \overline{F_2}), \quad && t >0,\\
    &F_1(1,\xi) = W_1 (\xi), \quad F_2 (1, \xi) = W_{2}(\xi), && \xi \in \re.
\end{aligned}
\right. 
\end{equation}
Notice that the system is the corresponding version of \eqref{ODEgen}.
\end{rem}

Since the function $F_2$ contains the factor $t^{\mu |W_1|^2}$,
the estimate \eqref{asymp2} suggests that the second component $u_2$ decays in time slower than a linear solution does.
However, in order to obtain a decay estimate from below, which makes the slow decay rigorous,
we need to show that there exists a set $\mathcal{O}\subset \re$ with positive measure such that
$W_1(\xi)\neq0$ and
\[
	(\mu + i\eta) W_2 (\xi) 
   - i\lambda_6 \tfrac{W_1(\xi)^2}{|W_1(\xi)|^2} \overline{W_2 (\xi)}
   \neq 0
\]
hold on $\mathcal{O}$.
One may expect that a generic solution possesses this property and
it would be actually possible to find a sufficient condition in terms of the initial data $(u_{1,0},u_{2,0})$ which ensures the existence of the above set $\mathcal{O}$.

We here take another way to prove the existence of a slowly-decaying solution.
Namely, 
we consider the final state problem:
%of the following system of cubic nonlinear Schr\"{o}dinger 
%(NLS) equations in one space dimension :
\begin{equation}\label{NLS110}
\left\{ 
\begin{aligned}
   & {{\mathcal L}}u_1 
       = 3 \lambda_1 |u_1|^2 u_1, \quad && t \in \re, x \in \re,\\
   & {{\mathcal L}}u_2 
       = \lambda_6 (2|u_1|^2 u_2 + u^2_1 \overline{u}_2), \quad && t \in \re, x \in \re,\\
    &\| u_j(t)-u_{\mathrm{ap},j}(t)\|_{L^\infty} = o(\|u_{\mathrm{ap},j}(t)\|_{L^\infty})\quad \text{as } t\to \infty,  && j=1,2,
\end{aligned}
\right. 
\end{equation}
where $u_j : \re \times \re \to \mathbb{C}$ ($j=1,2$) are unknown functions, 
and 
$\lambda_{1}$ and $\lambda_{6}$ are real constants 
satisfying \eqref{E:cond}.
%$(\lambda_6 - \lambda_1)(\lambda_6 - 3\lambda_1)<0$. 
Furthermore, $u_{\mathrm{ap},j}$ are given functions defined by
\begin{equation}\label{E:uapj}
u_{\mathrm{ap},j}
=M(t)D(t)F_{j}(t,\cdot)=t^{-\frac12}F_{j}(t,\tfrac{x}{t})e^{\frac{ix^{2}}{2t}-i\frac{\pi}{4}},
\qquad j=1,2,
\end{equation}
where $F_1$ and $F_2$ are defined as in \eqref{F11} and \eqref{F12}, respectively, from 
a pair of prescribed $\mathbb{C}$-valued functions $W_1$ and $W_2$.
%\begin{equation}
%F_{1}(t,\xi)=
%W_1(\xi)e^{-3i\lambda_{1}|W_1(\xi)|^{2}\log t}
%\label{F11}
%\end{equation}
%and
%\begin{equation}\label{F12}
%F_{2}(t,\xi)
%= \tilde{W}_2(t,\xi)e^{-3i\lambda_{1}|W_1(\xi)|^{2}\log t}
%\end{equation}
%with given final datum $W_1$ and $W_2$, where 
%$\tilde{W}_2$ is defined as in Theorem \ref{1.1}.
%$\mu = \sqrt{-3(\lambda_6 - \lambda_1)(\lambda_6 - 3\lambda_1)}$.
%For the final state problem (\ref{NLS11}) we have the following result. 

Our second result is as follows:
\begin{theo}\label{1.2}
%Suppose $(\lambda_6 - \lambda_1)(\lambda_6 - 3\lambda_1) < 0$. 
Let $W_1, W_2 \in H^2$.
Fix $\nu \in (1/2,1)$ and $\delta \in (0,\nu-1/2)$.
Then, if
$\|W_1\|_{H^2}$ is sufficiently small %Let $0<\gamma_1 < \gamma_2 < 1/100$. 
then there exist $T=T(\nu,\delta, \|W_2\|_{H^2}) \ge 2$ and
 a solution
$(u_1, u_2 )\in C([T, \infty); L^2(\R))^2$ to \eqref{NLS110} such that 
$(Ju_1, Ju_2 )\in C([T, \infty); L^2(\R))^2$ and
\begin{align}\label{E:T2_error2}
\|u_{1}(t)-u_{\mathrm{ap},1}(t)\|_{L_{x}^{2}}
\lesssim t^{-\nu},\quad
\|u_{2}(t)-u_{\mathrm{ap},2}(t)\|_{L_{x}^{2}}
\lesssim t^{-\nu+\delta}
\end{align}
in $L^{2}(\R)$ and
\begin{align}\label{E:T2_error}
\|u_{1}(t)-u_{\mathrm{ap},1}(t)\|_{L_{x}^{\infty}}
\lesssim t^{-\nu-\frac14},\quad
\|u_{2}(t)-u_{\mathrm{ap},2}(t)\|_{L_{x}^{\infty}}
\lesssim t^{-\nu-\frac14 +\delta}
\end{align}
in $L^{\infty}(\R)$ for $t \ge T$.
In particular, if $(W_1,W_2)$ is chosen so that
$W_1(\xi) \neq 0$ and
\begin{equation}\label{E:nonvanishing}
	(\mu + i\eta) W_2 (\xi) 
   - i\lambda_6 \tfrac{W_1(\xi)^2}{|W_1(\xi)|^2} \overline{W_2 (\xi)}
   \neq 0
\end{equation}
hold on $\re$ in addition then
\[
	\| u_2(t) \|_{L^\infty} \sim
	\|u_{\mathrm{ap},2}(t)\|_{L_{x}^{\infty}} \sim t^{-\frac12 + \mu \|W_1\|_{L^\infty}}.
\]
On the other hand, if $(W_1,W_2)$ is chosen so that
\begin{equation}\label{E:vanishing}
	(\mu + i\eta) W_2 (\xi) 
   - i\lambda_6 \tfrac{W_1(\xi)^2}{|W_1(\xi)|^2} \overline{W_2 (\xi)}
   = 0
\end{equation}
holds on $\re$ then \eqref{E:T2_error} implies
\[
	 u_2(t) = M(t) D(t) (t^{-\mu |W_1|^2} W_2 e^{- i3\lambda_1 |W_1|^2 \log t })
                   + O(t^{-\nu}) 
\]
as $t\to\infty$.
\end{theo}

\begin{rem}
A simple way to choose a pair of functions $(W_1,W_2)$ which satisfies the condition \eqref{E:nonvanishing} 
is to take a pair of non-zero real-valued functions.
For instance,  the choice $(W_1,W_2) = (\varepsilon e^{-x^2}, e^{-x^2})$ works for small $\varepsilon>0$.
%We also remark that if $\| W_1 \|_{L^\infty}$ is small compared with $\nu - 1/2$ then 
%\eqref{E:T2_error} gives us
%$\| u_j(t)-u_{\mathrm{ap},j}(t)\|_{L^\infty} = o(\|u_{\mathrm{ap},j}(t)\|_{L^\infty})$ for $j=1,2$.
\end{rem}

\begin{rem}
The upper bound of $\nu$ can be compared with the asymptotic analysis for the linear equation.
One has
\[
	\|U(t) u_0 - t^{-\frac12} e^{i\frac{x^2}{2t}-i \frac{\pi}4}  \mathcal{F} [u(0)](\tfrac{x}t)\|_{L^p} \lesssim  \|u_0\|_{H^{0,2}}  t^{-\frac54 + \frac1{2p}}
\]
as $t\to\infty$ for any $p\in[2,\infty]$. The order in the right hand side is optimal.
The estimates \eqref{E:T2_error2} and \eqref{E:T2_error} almost reach to this order.
They contain merely an epsilon loss.
\end{rem}

The rest of the paper is organized as follows.
We treat the Cauchy problem \eqref{NLS} in Section 2.
We mainly prove the global existence and the asymptotic behavior of the 
second component $ u_2 $ to show Theorem \ref{1.1}. 
In section 3, we treat the final state problem and give the proof of Theorem \ref{1.2}.

\section{Proof of Theorem \ref{1.1}}
%%%%%%%%%%%%%%%%%%%%%%%%%%%%%%%
%%%%%%%%%%%%%%%%%%%%%%%%%%%%%%%

In this section, we prove Theorem \ref{1.1}.

\subsection{On the first component $u_1$}
Note that $u_1$ satisfies the single equation
\begin{equation}
%i \partial_t u_1 + \frac{1}{2}\partial^2_x u_1 
     \mathcal{L}u_{1}  = 3 \lambda_1 |u_1|^2 u_1. \label{SS}
\end{equation}
As mentioned in the introduction, the asymptotic behavior is well-known.

Following the argument in Hayashi-Naumkin \cite{HN}, one shows that if $\varepsilon_1>0$ is sufficiently small then
(\ref{SS}) admits a unique time-global solution $u_1(t)$
such that 
\begin{align}
&\|u_{1}(t)\|_{H^1} + \|Ju_{1}(t)\|_{L^2} \le C \varepsilon_1 \langle t \rangle^{C_1 \varepsilon^2_1}, 
\label{est.u1.1}\\
&\|u_{1}(t)\|_{L^{\infty}} \le C \varepsilon_1 \langle t \rangle^{-1/2}.  
\label{est.u1.2}
\end{align}
Furthermore if we let $v_{1}:={{\mathcal F}}U(-t)u_{1}(t)$ %($v_{1}$の定義を変更しました). 
%$v_1 = v_1 (t, x) := U(-t)u_1 (t)$. 
then there exists some 
$\alpha = \alpha(\xi) \in L^{\infty} \cap L^2$ such that 
\begin{align}
&\|v_1 (t)\|_{L^{\infty}} + \|\alpha\|_{L^{\infty}}
\le C_2 \varepsilon_1, \label{est.u1.2.5} \\
&\| e^{3i\lambda_1 \Phi_1} v_1(t) - \alpha \|_{L^{\infty}} 
\le C \varepsilon^3_1 t^{-1/4 + 2C_1\varepsilon^2_1} \label{est.u1.3}
\end{align}
for $t \in [1, \infty)$, where $\Phi_1 = \Phi_1 (t, \xi): 
= \int^t_1 \tau^{-1} |v_1 (\tau)|^2 \ d\tau$. 

Let us deduce (\ref{asymp1}). To this end, we introduce one lemma.
\begin{lem}\label{lem:2.1}
One has $\mathcal{F} M \mathcal{F}^{-1} = U(-1/t)$. Further, for any $f \in {H}^1$,
\begin{equation}\label{E:U-1}
	\| (\mathcal{F} M \mathcal{F}^{-1} - 1)f\|_{L^\infty}
	\lesssim |t|^{-\frac14} \| f \|_{\dot{H}^1}
\end{equation}
for $t\neq0$.
\end{lem}
This is well-known. For reader's convenience, we give a proof.
\begin{proof}
The identity is obvious. Let us prove \eqref{E:U-1}. By Gagliardo-Nirenberg inequality,
\[
	\| (\mathcal{F} M \mathcal{F}^{-1} - 1)f\|_{L^\infty}
	\lesssim \| (\mathcal{F} M \mathcal{F}^{-1} - 1)f\|_{L^2}^\frac12
	\|  (\mathcal{F} M \mathcal{F}^{-1} - 1)f\|_{\dot{H}^1}^\frac12
\]
Since $U(-1/t)$ is unitary on $\dot{H}^1$, one has
$\|  (\mathcal{F} M \mathcal{F}^{-1} - 1)f\|_{\dot{H}^1} \le 2\|f\|_{\dot{H}^1}$.
Further, by using $|M-1|=|\sin (|x|^2/4t)| \le (|x|^2/4|t|)^{1/2}$, one has
\[
	\| (\mathcal{F} M \mathcal{F}^{-1} - 1)f\|_{L^2}
	\lesssim |t|^{-\frac12} \| f \|_{\dot{H}^1}.
\]
Hence, we obtain the result.
\end{proof}

By Lemma \ref{lem:2.1}, \eqref{est.u1.1}, 
(\ref{est.u1.3}), and the fact that $\|v_1\|_{\dot{H}^1} = \|Ju_1\|_{L^2}$,
we see that 
\begin{align}
u_1(t) 
%& = U(t)\mathcal{F}^{-1} v_1 (t) \nonumber \\
& = M(t)D(t) \mathcal{F} M(t) \mathcal{F}^{-1}v_1(t) \nonumber \\
& = M(t)D(t) v_1 (t) + O(t^{-3/4 + C_1 \varepsilon^2_1}) \nonumber \\
& = M(t) D(t) e^{-3i\lambda_1 \Phi_1 (t)} \alpha + O(t^{-3/4 + 2C_1 \varepsilon^2_1}) 
\label{asymp.u1}
\end{align}
in $L^{\infty}$ as $t \to \infty$. As for the phase correction, we write 
\begin{align*}
\Phi_1 (t) 
&= \int^t_1 \tau^{-1} |v_1(\tau)|^2 \ d\tau  \\
&= |\alpha|^2 \log t 
+ \int^t_1 \tau^{-1} (|e^{3i\lambda_1 \Phi_1}v_1(\tau)|^2 - |\alpha|^2) \ d\tau.
\end{align*}
By virtue  of (\ref{est.u1.2.5}) and (\ref{est.u1.3}), the above integral converges to some 
$\theta_1 = \theta_1(\xi)$ in $L^{\infty}$ as $t \to \infty$, and hence we have 
\begin{align*}
\Phi_1 (t) 
&= |\alpha|^2 \log t  + \theta_1 + O(t^{-1/4 + C_1 \varepsilon^2_1})
\end{align*}
in $L^{\infty}$ as $t \to \infty$. Plugging it into (\ref{asymp.u1}) and 
setting $W_1 := e^{-3i \lambda_1 \theta_1} \alpha$, we have 
\begin{align*}
u_1(t) 
&= M(t) D(t) e^{-3i\lambda_1 |\alpha|^2 \log t} W_1 + O(t^{-3/4 + C_1 \varepsilon^2_1}) \\
&= M(t) D(t) e^{-3i\lambda_1 |W_1|^2 \log t} W_1 + O(t^{-3/4 + C_1 \varepsilon^2_1}),
\end{align*}
which is  (\ref{asymp1}). 
Further, we see from \eqref{est.u1.3} that
\begin{equation}\label{est.u1.4}
	\| e^{3i\lambda_1 (\Phi_1- \theta_1)} v_1(t) - W_1 \|_{L^{\infty}} 
\le C \varepsilon_1 t^{-1/4 + 2C_1\varepsilon^2_1}.
\end{equation}

\subsection{Global bound of the solution}
Let us move on to the analysis of the second component $u_2$.
Due to a standard local theory, one obtains a time-local solution
\[
	u_2 \in C(I; H^1 \cap H^{0,1} ), \quad 0 \in \exists I \subset \R.
\] 
See \cite{Cazenave}, for instance.
Notice that the second equation of the system \eqref{NLS} is a linear equation in $u_2$.
Hence, it is obvious that if $u_1$ exists globally in time then so does $u_2$. 

The main issue in this subsection is to obtain the bounds (\ref{upper_bound}) on $u_2$.
%Invoking (\ref{est.u1.1})-(\ref{est.u1.3}), we proceed in the estimates of $u_2$. 
Recall that $\varepsilon_2:=\|u_{2,0}\|_{H^2} +\|u_{2,0}\|_{H^{0,1}}$.
% and let 
%$T^{\ast}>0$ be the life span of the solution $u_2$. Precisely speaking, it is defined as 
For $\delta \in (0,1/5)$, we introduce
\begin{align}
T^{\ast} 
:= \sup \left\{ T > 0 : \sup_{0\le t <T} \langle t \rangle^{-\delta}
     (\|u_2 (t)\|_{H^1} + \|Ju_2 (t)\|_{L^2}) < 4 \varepsilon_2 \right\}.
\label{lifespan}
\end{align}
The main step of the proof is to show $T^\ast = \infty$.

Before proving this, let us collect basic facts on $u_2$.
Note that $u_2$ solves
\begin{align}\label{E:u2int}
u_2(t) = U(t-t_0)u_2(t_0) - i \lambda_6 \int^t_{t_0} U(t-\tau) 
\mathcal{N}_2(u_1 (\tau), u_2(\tau)) \ d\tau
\end{align}
for any fixed $t_0 \in \R$,
where 
\[
	\mathcal{N}_2(u_1, u_2) = 2|u_1|^2 u_2 + u^2_1 \overline{u}_2.
\]
One has
\begin{align}\label{E:u2intJ}
 J u_2(t) = U(t-t_0)Ju_2(t_0) - i \lambda_6 \int^t_{t_0} U(t-\tau) J
\mathcal{N}_2(u_1, u_2)(\tau) \ d\tau.
\end{align}
The following identity is useful:
\begin{equation}\label{E:JN2}
\begin{aligned}
J\mathcal{N}_2(u_1, u_2) 
& = 2|u_1|^2 Ju_2 - u^2_1 \overline{Ju_2} \\
& \quad +2 (\overline{u}_1 u_2 Ju_1 - u_1 u_2 \overline{Ju_1}) 
   + 2 u_1 \overline{u}_2 Ju_1.
\end{aligned}
\end{equation}

Let us begin with the following lemma.
\begin{lem}
For any $T_0>1$, there exists $\varepsilon_1^\ast>0$ such that if $\varepsilon_1 \in (0, \varepsilon_1^\ast)$ then
\[
	\sup_{t\in [0,T_0]}(\|u_2 (t)\|_{H^1} + \|Ju_2(t)\|_{L^2} ) \le 2 \varepsilon_2.
\]
In particular,
$T^{\ast}>T_0$ for such $\varepsilon_1$, where $T^\ast$ is given in \eqref{lifespan}.
\end{lem}
\begin{proof}
Plugging (\ref{est.u1.1}) and (\ref{est.u1.2}) to 
\eqref{E:u2int} and \eqref{E:u2intJ} with $t_0=0$,
we have  
\begin{align*}
&\|u_2 (t)\|_{H^1} + \|Ju_2(t)\|_{L^2} \\
&\quad \le \varepsilon_2 + C\int^t_0 (\|u_1(\tau)\|_{H^1}+\|Ju_1(\tau)\|_{L^2})^2 
     (\|u_2 (\tau)\|_{H^1} + \|Ju_2(\tau)\|_{L^2} ) \ d\tau \\
&\quad \le \varepsilon_2 + C\varepsilon^2_1 \int^t_0 \langle \tau 
\rangle^{2C_1 \varepsilon^2_1} 
     (\|u_2 (\tau)\|_{H^1} + \|Ju_2(\tau)\|_{L^2} ) \ d\tau,
\end{align*}
where we have used \eqref{E:JN2} to estimate the nonlinearity.
Hence, Gronwall's inequality yields 
\begin{equation}\label{E:u2bound0to1}
	\begin{aligned}
\|u_2 (t)\|_{H^1} + \|Ju_2(t)\|_{L^2} 
 &\le (1+C\varepsilon^2_1 \langle T_0 \rangle^{1+2C_1 \varepsilon_1^2}e^{C\varepsilon^2_1 \langle T_0 \rangle^{1+2C_1 \varepsilon_1^2}}) \varepsilon_2\\
 &<2\varepsilon_2
\end{aligned}
\end{equation}
for $t \in [0, T_0]$ if $\varepsilon_1$ is sufficiently small. 
Together with the definition of $T^\ast$, one
obtains $T^\ast >T_0$ as desired.
%We will consider the estimates 
%of $u_2(t)$ for $t \in [0, T^{\ast})$ first, and, in next section, show that $T^{\ast}$ is 
%actually chosen as $\infty$. 
%We remark that we have used the identities
%to estimate $\| Ju_2(t) \|_{L^2}$.
\end{proof}

We introduce notation.
Let $v_2 (t) := {{\mathcal F}}U(-t) u_2$, or equivalently,
\begin{equation}\label{E:u2v2}
	u_2(t) = M(t) D(t) \mathcal{F} M(t) \mathcal{F}^{-1} v_2(t).
\end{equation}
We further introduce
\begin{align*} 
w(t,\xi) := e^{3i\lambda_1 (\Phi_1(t)-\theta_1)}v_2(t). 
\end{align*}
We have the following formula for $w$:
\begin{lem}
The following formula for $w(t)$ is valid:
\begin{align}
\begin{pmatrix}
w(t) \\
\overline{w(t)}
\end{pmatrix}
& = P Q(t) P^{-1}
\begin{pmatrix}
w(1) \\
\overline{w(1)}
\end{pmatrix}
\nonumber \\
& \quad + P Q(t)
\int^t_1 Q(\tau)^{-1}P^{-1} 
\begin{pmatrix}
S_2 + e^{3i\lambda_1 (\Phi_1-\theta_1)} R_2 \\
\overline{S_2 + e^{3i\lambda_1 (\Phi_1-\theta_1)} R_2 }
\end{pmatrix}(\tau)
\ d \tau \label{matrix3}
\end{align}
for $t \ge1$, where 
$e^{i\theta} = W_1/|W_1|$,
\begin{align*}
P ={}& \begin{pmatrix}
\lambda_6 e^{2i\theta} & \eta - i\mu \\
\eta - i\mu & \lambda_6 e^{-2i\theta}
\end{pmatrix}, \\
P^{-1} ={}& \frac1{2i\mu(\eta-i\mu)} \begin{pmatrix}
\lambda_6 e^{-2i\theta} &- (\eta - i\mu) \\
-(\eta - i\mu) & \lambda_6 e^{2i\theta}
\end{pmatrix},\\
Q(t) ={}& 
\begin{pmatrix}
t^{-\mu |W_1|^2} & 0 \\
0 & t^{\mu |W_1|^2} 
\end{pmatrix},
\end{align*}
for $\xi\in\R$ such that $W_1(\xi)\neq0$,
and
\begin{align}
R_2(t) 
%&= -i\lambda_6 t^{-1} \mathcal{F}(M(t)^{-1}-1) \mathcal{F}^{-1} 
%                   \mathcal{N}_2 (\mathcal{F}M{{\mathcal F}}^{-1}v_1, 
%                   \mathcal{F}M{{\mathcal F}}^{-1}v_2) (t) \nonumber \\
%& \quad -i \lambda_6 t^{-1} ( \mathcal{N}_2 (\mathcal{F}M{{\mathcal F}}^{-1}v_1, \mathcal{F}M{{\mathcal F}}^{-1}v_2)(t) - \mathcal{N}_2 (v_1,v_2)(t) ), \label{R2}
&= -i\lambda_6 t^{-1} \mathcal{F}M(t)^{-1} \mathcal{F}^{-1} 
                   \mathcal{N}_2 (\mathcal{F}M{{\mathcal F}}^{-1}v_1, 
                   \mathcal{F}M{{\mathcal F}}^{-1}v_2) (t) \nonumber \\
& \quad +i \lambda_6 t^{-1} \mathcal{N}_2 (v_1,v_2)(t) , \label{R2}
\end{align}
and
\begin{align}
S_2(t) 
& =  i\eta t^{-1} (|e^{3i\lambda_1 (\Phi_1(t)-\theta_1)}v_1(t)|^2 - |W_1|^2) w (t)
      \nonumber \\
& \qquad -i\lambda_6 t^{-1} \{(e^{3i\lambda_1 (\Phi_1(t)-\theta_1)} v_1(t))^2 - W_1^2\}
           \overline{w (t)}. \label{S2}
\end{align}
The identity is valid also on $\{W_1(\xi) = 0\}$ with the convention $PQ(t) P^{-1} = I_2$,
where $I_2$ is the identity matrix $\mathrm{diag}(1,1)$.
\end{lem}
\begin{proof}
%Note that $u_2$ satisfies the integral equation: 
%Then we see that 
%Since 
%\begin{align*}
%Ju_2(t) = U(t) x u_2(1) - i \lambda_6 
%\int^t_1 U(t-\tau) J \mathcal{N}_2 (u_1(\tau), u_2(\tau)) \ d\tau. 
%\end{align*}
%and
By definition of $R_2$, one sees that $v_2$ satisfies 
\begin{align*}
i\partial_t v_2 
&=  \lambda_6 t^{-1} \mathcal{F} M(t)^{-1} \mathcal{F}^{-1} 
      \mathcal{N}_2(\mathcal{F}M(t){{\mathcal F}}^{-1}v_1, 
      \mathcal{F}M(t){{\mathcal F}}^{-1}v_2) \nonumber \\
&=  \lambda_6 t^{-1} \mathcal{N}_2(v_1,v_2) 
     + iR_2(s). %\label{eq.Fv2}
\end{align*}
%where 
%By (\ref{est.u1.3}), we know that $e^{3i\lambda_1 \Phi_1 (t)} v_1 
%\to \alpha$ as $t \to \infty$ in $L^{\infty}$, 
%where $\Phi_1 (t) = \int^t_1 \tau^{-1} |v_1|^2 \ d\tau$.
Hence, recalling the definition of  $\mathcal{N}_2$, 
one has
%\begin{align}
%i \partial_t v_2 
%=  \lambda_6 t^{-1} \left\{2|v_1|^2 v_2 
%+ %e^{3i\lambda_1 (\Phi_1(t)-\theta_1)} (
%v_1^2 
%   %)         e^{-6i\lambda_1 (\Phi_1(t)-\theta_1)} 
%   \overline{v_2} \right\}
%+ iR_2(t). 
%\label{eq.Fv2.2}
%\end{align}
%%We here consider a gauge-transform of $v_2$, i.e., we let
%%\begin{align*} 
%%w(t,\xi) := e^{3i\lambda_1 \Phi_1(t)}v_2(t). 
%%\end{align*}
%%Then $e^{-6i \lambda_1 \Phi_1(t)}$ in (\ref{eq.Fv2.2}) will drop out. In fact 
%%we have 
%This shows
\begin{align}
i\partial_t w 
& = e^{3i\lambda_1( \Phi_1 - \theta_1)} i\partial_t v_2 - 3 \lambda_1 t^{-1} |v_1|^2 w \nonumber\\
& = -\eta t^{-1} |v_1|^2 w 
       +\lambda_6 t^{-1} (e^{3i\lambda_1 (\Phi_1-\theta_1)}v_1)^2 \overline{w}
       + ie^{3i\lambda_1 (\Phi_1-\theta_1)} R_2(t) \nonumber \\
&=  -\eta t^{-1} |W_1|^2 w 
       + \lambda_6 t^{-1} W_1^2 \overline{w}
       + i(S_2(t) + e^{3i\lambda_1 (\Phi_1-\theta_1)} R_2(t)), \label{eq.w}
\end{align}
where $\eta=3\lambda_1 - 2\lambda_6$.

If $W_1(\xi_0)=0$ then we have
\[
	w(t,\xi_0) = w(1,\xi_0) + \int_1^t (S_2 + e^{3i\lambda_1 (\Phi_1-\theta_1)} R_2)(\tau,\xi_0) d\tau,
\]
which gives us \eqref{matrix3} under the convention $PQ(t) P^{-1} = I_2$.

In the sequel we suppose that $W_1(\xi_0)\neq0$. For simplicity, we omit $\xi_0$.
We write (\ref{eq.w}) in a matrix form: 
\begin{align*}
\partial_t  
\begin{pmatrix}
w \\
\overline{w}
\end{pmatrix}
&= it^{-1} |W_1|^2 
\begin{pmatrix}
\eta & - \lambda_6 e^{2i\theta}\\
\lambda_6 e^{-2i\theta} & - \eta
\end{pmatrix}
\begin{pmatrix}
w \\
\overline{w}
\end{pmatrix}
+ \begin{pmatrix}
S_2 + e^{3i\lambda_1 (\Phi_1-\theta_1)} R_2 \\
\overline{S_2 + e^{3i\lambda_1 (\Phi_1-\theta_1)} R_2}
\end{pmatrix}, %\label{matrix}
\end{align*}
where $e^{i\theta}=W_1/|W_1|$.
The $2 \times 2$ matrix on the first term of the right hand side 
%of (\ref{matrix})
 possesses the eigenvalues $\pm i \mu$. %, where $\mu=\sqrt{\lambda_6^2 - \eta^2}$.
Further, the eigenvectors associated with the eigenvalue $\pm i\mu$ are
$(\lambda_6 e^{2i\theta} \quad \eta - i \mu)^T$ 
and  
$(\eta- i \mu \quad \lambda_6 e^{-2i\theta})^T$, respectively.
% where 
%$\mu = \sqrt{-3(3\lambda_1 - \lambda_6)(\lambda_1 - \lambda_6)}$. 
Then, from the diagonalization of the matrix, it follows that 
\begin{align*}
\partial_t 
\begin{pmatrix}
w \\
\overline{w}
\end{pmatrix}
&= it^{-1} |W_1|^2 P
\begin{pmatrix}
i\mu & 0 \\
0 & - i\mu
\end{pmatrix}
P^{-1} 
\begin{pmatrix}
w \\
\overline{w}
\end{pmatrix}
+ 
\begin{pmatrix}
S_2 + e^{3i\lambda_1 (\Phi_1-\theta_1)} R_2 \\
\overline{S_2 + e^{3i\lambda_1 (\Phi_1-\theta_1)} R_2 }
\end{pmatrix}.
%. \label{matrix2}
\end{align*}  
%where 
We further rewrite it %(\ref{matrix2})
 in such a way that 
\begin{align}
&\partial_t 
%\begin{pmatrix}
%t^{\mu |\alpha|^2} & 0 \\
%0 & t^{-\mu |\alpha|^2} 
%\end{pmatrix}
Q(t)^{-1} P^{-1} 
 \begin{pmatrix}
w \\
\overline{w}
\end{pmatrix}
%\nonumber \\
%&\quad 
= 
%\begin{pmatrix}
%t^{\mu |\alpha|^2} & 0 \\
%0 & t^{-\mu |\alpha|^2} 
%\end{pmatrix}
Q(t)^{-1} P^{-1}
\begin{pmatrix}
S_2(t) + e^{3i\lambda_1 (\Phi_1-\theta_1)} R_2(t) \\
\overline{S_2(t) + e^{3i\lambda_1 (\Phi_1-\theta_1)} R_2(t) }
\end{pmatrix}
. \label{relation.w}
\end{align}
An integration in time gives us the desired formula \eqref{matrix3}.
\end{proof}

%%%%%%%%%%%%%%%%%%%%%%%%%%%%%%%
%%%%%%%%%%%%%%%%%%%%%%%%%%%%%%%
%\section{Asymptotic behavior of solutions}
%%%%%%%%%%%%%%%%%%%%%%%%%%%%%%%
%%%%%%%%%%%%%%%%%%%%%%%%%%%%%%%
%In this section, we complete the proof of Theorem \ref{1.1} by establishing 
%the asymptotic estimates.
%To this end, 

%We will obtain time-global estimates of $\|Ju_2(t)\|_{L^2}$ 
%by making use of (\ref{est.Ju2.2}), (\ref{est.u2.infty}) and (\ref{est.Fv2.infty}).
%We collect two lemmas concerning the estimates of $S_2$ and $R_2$ in (\ref{est.Fv2.infty}). 

We now estimate the error terms $R_2$ and $S_2$.

\begin{lem}\label{3.1}
Let %$\varepsilon_2=\|u_{2,0}\|_{H^1}+\|u_{2,0}\|_{H^{0,1}}$ 
%and 
$S_2$ be defined in 
$(\ref{S2})$. Then there exists some positive constant $C$ %independent of $T^{\ast}$ 
such that 
\begin{align}
\|S_2 (t)\|_{L^{\infty}} 
\le C \varepsilon^2_1  \ t^{-\frac54 + 2C_1 \varepsilon^2_1} \|v_2(t)\|_{L^\infty}
%            (\|u_2(\sigma)\|_{L^2} + \|Ju_2 (\sigma)\|_{L^2}).
\end{align}
for any $t \ge1$.
\end{lem}
\begin{proof}
It is obvious by
\begin{align*}
	\|S_2(t)\|_{L^\infty} \lesssim t^{-1} \| v_1e^{3i\lambda_1 (\Phi_1-\theta_1)} -W_1 \|_{L^\infty}
	(\|v_1(t) \|_{L^\infty} + \|W_1 \|_{L^\infty}) \|v_2(t)\|_{L^\infty}
\end{align*}
and \eqref{est.u1.2.5} and \eqref{est.u1.4}.
\end{proof}

\begin{lem}\label{3.2}
Let $\varepsilon_2=\|u_{2,0}\|_{H^1}+\|u_{2,0}\|_{H^{0,1}}$ and $R_2$ 
be defined in 
$(\ref{R2})$. Then there exists some positive constant $C$ %independent of $T^{\ast}$ 
such that 
\begin{align}
\|R_2 (t)\|_{L^{\infty}} 
\le C \varepsilon^2_1  \ t^{-\frac54 + 2C_1 \varepsilon^2_1 } \|v_2(t)\|_{H^1}
%     (\|u_2(\sigma)\|_{L^2} + \|Ju_2 (\sigma)\|_{L^2}). 
\end{align}
for any $t \ge 1$.
\end{lem}
\begin{proof}
We have
\begin{align*}
R_2
&= -i\lambda_6 t^{-1} \mathcal{F}(M^{-1}-1) \mathcal{F}^{-1} 
                   \mathcal{N}_2 (\mathcal{F}M{{\mathcal F}}^{-1}v_1, 
                   \mathcal{F}M{{\mathcal F}}^{-1}v_2) \\
& \quad -i \lambda_6 t^{-1} ( \mathcal{N}_2 (\mathcal{F}M{{\mathcal F}}^{-1}v_1, \mathcal{F}M{{\mathcal F}}^{-1}v_2) - \mathcal{N}_2 (v_1,\mathcal{F}M{{\mathcal F}}^{-1} v_2) ) \\
& \quad -i \lambda_6 t^{-1} ( \mathcal{N}_2 (v_1, \mathcal{F}M{{\mathcal F}}^{-1}v_2) - \mathcal{N}_2 (v_1,v_2))\\
& =: R_{21}  + R_{22} + R_{23}. 
\end{align*}
Combining \eqref{E:U-1},
the fact that $H^1(\re)$ is an algebra and that $U(-1/t)$ is unitary on $H^1(\re)$, we have
\begin{align*}
	\| R_{21} \|_{L^\infty} \lesssim{}& t^{-\frac54}  
	\| \mathcal{N}_2 (U(-1/t)v_1, U(-1/t)v_2)\|_{H^1} \\
	\lesssim{}& t^{-\frac54 }  
	\|U(-1/t)v_1\|_{H^1}^2 \|U(-1/t)v_2\|_{H^1} 
	=	t^{-\frac54 }  
	\|v_1\|_{H^1}^2 \|v_2\|_{H^1}
\end{align*}
for $t\ge1$.
Similarly,
%\begin{align*}
%	&\| R_{22} \|_{L^\infty} \\
%	&{}\lesssim t^{-1}  
%	\| (U(-1/t)-1) v_1\|_{L^\infty}(\| U(-1/t) v_1\|_{L^\infty} + \| v_1\|_{L^\infty})\| U(-1/t) v_2\|_{L^\infty} \\
%	&{}\lesssim t^{-\frac54 + \frac{\delta}2}  
%	\|v_1\|_{H^1}^2 \|v_2\|_{H^1}.
%\end{align*}
%and
%\begin{align*}
%	\| R_{23} \|_{L^\infty}
%	&{}\lesssim t^{-1}  
%	\| v_1\|_{L^\infty}^2 \| (U(-1/t)-1) v_2\|_{L^\infty} \\
%	&{}\lesssim t^{-\frac54 + \frac{\delta}2}  
%	\|v_1\|_{H^1}^2 \|v_2\|_{H^1}.
%\end{align*}
one has
\[
	\| R_{22} \|_{L^\infty} + \| R_{23} \|_{L^\infty} \lesssim
	t^{-\frac54}  
	\|v_1\|_{H^1}^2 \|v_2\|_{H^1}.
\]
Hence, using the fact that $\|v_1\|_{H^1}^2 = \|u_1\|_{L^2}^2 + \|Ju_1\|_{L^2}^2$ and 
the bound \eqref{est.u1.1} for $u_1$, we obtain the result.
\end{proof}
%Lemma \ref{3.1} and \ref{3.2} will be proved in the next subsection. 

Now we are in the position to complete the proof of the global bound.
\begin{lem}
$T^\ast = \infty$ for sufficiently small $\varepsilon_1$, where $T^\ast$ is given in \eqref{lifespan}.
\end{lem}
\begin{proof}
We prove the lemma by contradiction. 
Fix $\delta\in (0,1/5)$ and
assume that $T^{\ast}$ is finite for any small $\varepsilon_1>0$. 
Then, we have a bound
\begin{equation}\label{E:bound_pf_asmp}
	\|u_2(t)\|_{H^1} + \| Ju_2 (t)\|_{L^2} < 4 \varepsilon_2 \langle t\rangle^{\delta}
\end{equation}
for $t\in (0,T^\ast)$. Further, by continuity of the solution, one has
\[
\|u_2(T^{\ast})\|_{H^1} + \| Ju_2 (T^{\ast})\|_{L^2} = 4 \varepsilon_2 \langle T^{\ast}\rangle^{\delta}.
\]
To obtain a contradiction, we shall show \eqref{E:bound_pf_asmp} gives us a better bound
\begin{equation}\label{E:bound_pf_goal}
	\|u_2(t)\|_{H^1} + \| Ju_2 (t)\|_{L^2} \le 3 \varepsilon_2 \langle t\rangle^{\delta}
\end{equation}
for $t\in (0,T^\ast)$
and sufficiently  small $\varepsilon_1$.
Once \eqref{E:bound_pf_goal} is established, by letting $t \uparrow T^{\ast}$ in \eqref{E:bound_pf_goal},
%in (\ref{final.est.}), 
we see that $4 \varepsilon_2 \le 3 \varepsilon_2$. 
This is 
the contradiction, and we have $T^{\ast} = \infty$. 
%Hence, we consider $t>2$.

Let $T_0>1$ be a number to be chosen later.
Note that if $\varepsilon_1<\varepsilon_1^\ast(T_0)$  then
$T^\ast > T_0$ and
\eqref{E:u2bound0to1} gives us the desired bound
for $t\in [0,T_0]$.

We obtain the estimates for large $t$.
%Let us estimate $\| Ju_2 (t)\|_{L^2}$.
%Let $\varepsilon_2=\|u_{2,0}\|_{H^1}+\|u_{2,0}\|_{H^{0,1}}$. 
We deduce from \eqref{E:u2intJ} with $t_0=1$ that
\begin{equation}\label{est.Ju2}
\begin{aligned}
{\|Ju_2(t)\|_{L^2}}
\le{}& \|Ju_2(1)\|_{L^2} \\
 &{}+ C \int^t_1 (\|u_1\|^2_{L^{\infty}} \|Ju_2\|_{L^2}
  + \|u_1\|_{L^{\infty}} \|u_2\|_{L^{\infty}} \|Ju_1\|_{L^2})(\tau) \ d\tau. 
\end{aligned}
\end{equation}
Similarly,
\begin{equation}\label{est.H1u2}
\begin{aligned}
{\|u_2(t)\|_{H^1}}
\le{}& \|u_2(1)\|_{H^1} \\
 &{}+ C \int^t_1 (\|u_1\|^2_{L^{\infty}} \|u_2\|_{H^1}
  + \|u_1\|_{L^{\infty}} \|u_2\|_{L^{\infty}} \|u_1\|_{H^1})(\tau) \ d\tau. 
\end{aligned}
\end{equation}
Applying (\ref{est.u1.1}), (\ref{est.u1.2}), and \eqref{E:u2bound0to1} to (\ref{est.Ju2}) and (\ref{est.H1u2}), we see that 
\begin{align}
{\|u_2(t)\|_{H^1}}
+\|Ju_2(t)\|_{L^2} 
&\le 2\varepsilon_2 + C \varepsilon^2_1 \int^t_1 \tau^{-1} ({\|u_2(\tau)\|_{H^1}}
+\|Ju_2(\tau)\|_{L^2}) \ d\tau
\nonumber \\
& \qquad + C \varepsilon^2_1 \int^t_1 \tau^{-1/2 + C_1\varepsilon_1^2} 
       \|u_2(\tau)\|_{L^{\infty}} \ d\tau \label{est.Ju2.2}
\end{align}
for $t \ge 1$ if $\varepsilon_1 \le \varepsilon_1^\ast(1)$.

We next claim that \eqref{E:bound_pf_asmp} implies
that
\begin{align}
\|u_2 (t)\|_{L^{\infty}} 
&\le C \varepsilon_2 (1 + \varepsilon^2_1) t^{-1/2 + \mu C_2^2\varepsilon_1^2}
\label{est.u2.infty.2}
\end{align}
for $t \in [1,T^\ast)$.
%Our interest is now to see how to estimate $\|u_2 (t)\|_{L^{\infty}}$. 
%It is reduced to the bound of  $\|v_2 (t)\|_{L^{\infty}}$.
%This is because, one has
It follows from \eqref{E:bound_pf_asmp} that
\begin{align}
\|u_2 (t)\|_{L^{\infty}} 
%&= \|U(t) {{\mathcal F}}^{-1}v_2 (t)\|_{L^{\infty}} \nonumber \\
&= \|M(t)D(t)\mathcal{F}M(t) {{\mathcal F}}^{-1}v_2(t)\|_{L^{\infty}} \nonumber \\
&\le t^{-1/2} \|v_2(t)\|_{L^{\infty}} 
       + t^{-1/2} \|\mathcal{F} (M(t)-1) {{\mathcal F}}^{-1}v_2(t)\|_{L^{\infty}} \nonumber \\
&\le t^{-1/2} \|v_2(t)\|_{L^{\infty}} 
        + Ct^{-3/4} \|Ju_2 (t)\|_{L^2} \nonumber \\
&\le t^{-1/2} \|v_2(t)\|_{L^{\infty}} 
        + C \varepsilon_2 \ t^{-3/4 + \delta}. 
\label{est.u2.infty}
\end{align}
on the interval $[1,T^\ast)$.
The second term is acceptable since $\delta< 1/4$.
We estimate $\| v_2(t)\|_{L^\infty}= \|w(t)\|_{L^{\infty}}$.
By 
\[
	Q(t)Q(\tau)^{-1} = Q(t\tau^{-1})
\]
for $1 \le \tau \le t$, (\ref{matrix3}) and \eqref{est.u1.2.5}  yield
\begin{align*}
\|v_2 (t)\|_{L^{\infty}} 
\le C t^{\mu C_2^2 \varepsilon_1^2} \|w(1)\|_{L^{\infty}} 
    + C t^{\mu C_2^2 \varepsilon_1^2}\int^t_1 
          (\|S_2(\tau)\|_{L^{\infty}} + \|R_2(\tau)\|_{L^{\infty}}) \ d\tau. 
%          \label{est.Fv2.infty}
\end{align*}
Further, applying 
Lemmas \ref{3.1} and \ref{3.2}, the embedding $H^1 \hookrightarrow L^\infty$, 
the identities $\|v_2(t)\|_{L^2}=\|u_2\|_{L^2}$ and $\|v_2\|_{\dot{H}^1} = \|Ju_2\|_{L^2}$,
and \eqref{E:bound_pf_asmp},
we have 
\begin{align}
\|v_2 (t)\|_{L^{\infty}} 
& \lesssim  \varepsilon_2 t^{\mu C_2^2 \varepsilon_1^2}
    +  \varepsilon^2_1 \varepsilon_2 \ t^{\mu C_2^2 \varepsilon_1^2}\int^t_1 
          \tau^{-5/4 + 2C_1 \varepsilon^2_1 + \delta} \ d\tau \nonumber \\
& \lesssim \varepsilon_2 t^{\mu C_2^2 \varepsilon_1^2} +  \varepsilon^2_1 \varepsilon_2 \ 
t^{\mu C_2^2 \varepsilon_1^2} 
\label{est.Fv2.infty.2}
\end{align}
for $t \in [1, T^{\ast})$, if $\varepsilon_1>0$ is taken so that $2C_1 \varepsilon^2_1 \le 1/40$. 
Note that the implicit constants are independent of $\varepsilon_1$ and $\delta$ since
$2C_1 \varepsilon^2_1 +\delta \le 1/40+1/5<1/4$.
Plugging (\ref{est.Fv2.infty.2}) to (\ref{est.u2.infty}),
we reach to the claim \eqref{est.u2.infty.2}.
%\begin{align*}
%\|u_2 (t)\|_{L^{\infty}} 
%&\le C \varepsilon_2 (1 + \varepsilon^2_1) t^{-1/2 + \mu |\alpha|^2}
%+ C\varepsilon_2 t^{-3/4 + \delta} \nonumber \\
%&\le C \varepsilon_2 (1 + \varepsilon^2_1) t^{-1/2 + \mu |\alpha|^2}
%%\label{est.u2.infty.2}
%\end{align*}
%since $\delta < 1/4$. 

Applying the bound (\ref{est.u2.infty.2}) to the third term on the right hand side 
of (\ref{est.Ju2.2}), and we have 
\begin{align*}
{\|u_2(t)\|_{H^1}}+
\|Ju_2(t)\|_{L^2} 
\le{}& 2 \varepsilon_2 
+ C_3 \varepsilon_2 
t^{(C_1 +C^2_2\mu)\varepsilon^2_1}\\
&+ C_4 \varepsilon^2_1 \int^t_1 \tau^{-1} ({\|u_2(\tau)\|_{H^1}}
+\|Ju_2(\tau)\|_{L^2}) \ d\tau .
\end{align*}
Without loss of generality, one may suppose that
$C_4 \ge 2(C_1+C_2^2 \mu)$.
Then, Gronwall's inequality yields 
\begin{align}
{\|u_2(t)\|_{H^1}}+\|Ju_2(t)\|_{L^2} 
\le {}& 2 \varepsilon_2 
+ C_3 \varepsilon_2 
t^{(C_1 +C^2_2\mu)\varepsilon^2_1}\nonumber\\
&+ \int_1^t (2 \varepsilon_2 
+ C \varepsilon_2 
\tau^{(C_1 +C^2_2\mu)\varepsilon^2_1}) C_4 \varepsilon_1^2 \tau^{-1} (t/\tau)^{C_4 \varepsilon_1^2}d\tau\nonumber\\
={}& 2\varepsilon_2 t^{C_4 \varepsilon_1^2}
- \frac{C_1+ C_2^2\mu}{C_4-(C_1+ C_2^2\mu)} C_3 \varepsilon_2 
t^{(C_1 +C^2_2\mu)\varepsilon^2_1} \nonumber\\
&{}+ \frac{C_4}{C_4-(C_1+ C_2^2\mu)} C_3 \varepsilon_2 
t^{C_4\varepsilon^2_1} \nonumber\\
\le{}& (2+ C_3/2) \varepsilon_2 t^{C_4\varepsilon^2_1} \label{E:u2_final_est}
\end{align}
for $t\ge1$.

We first let $\varepsilon_1$ so small that $C_4 \varepsilon_1^2 \le \delta/2$ holds.
Then, we next choose $T_0>1$ so that
$(2+ C_3/2) T_0^{\delta/2}\le 3 T_0^{\delta}$.
If $\varepsilon_1 < \varepsilon_1^\ast(T_0)$, we 
see from \eqref{E:u2bound0to1} that
$T^\ast>T_0$ and
\[
\sup_{t\in[0,T_0]}(\|u_2 (t)\|_{H^1} + \|Ju_2(t)\|_{L^2} )
 \le 2\varepsilon_2.
\]
On the other hand, thanks to the choice of $T_0$, \eqref{E:u2_final_est} gives us
\[
	\sup_{t\in [T_0,T^\ast)}	\langle t \rangle^{-\delta}(\|u_2 (t)\|_{H^1} + \|Ju_2(t)\|_{L^2})
		\le 3 \varepsilon_2 .
\]
Thus, we obtain \eqref{E:bound_pf_goal}.
%
%In the same way, we have 
%\begin{align}
%\|u_2(t)\|_{H^1} 
%\le 3 \varepsilon_2 t^{(C_3 +C_1 +C^2_2 \mu) \varepsilon^2_1}\log t.  
%\label{H^1.est.}
%\end{align}
%Combining (\ref{weight.est.}) and (\ref{H^1.est.}) together with 
%$\delta = (C +C_1 +C^2_2 \mu) \varepsilon^2_1$, we have 
%\begin{align}
%\langle t \rangle^{-\delta} (\|u_2(t)\|_{H^1} + \|Ju_2(t)\|_{L^2}) 
%\le 8 \varepsilon_2 \label{final.est.}
%\end{align}
%for $t \in [1,T^\ast)$.
%
\end{proof}

Le us complete the proof of the bound \eqref{upper_bound}.
If $\varepsilon_1$ is sufficiently small, we have 
$T^\ast = \infty$.
Then, the claim \eqref{est.u2.infty.2} holds for $t\in [1,\infty)$.
Further, arguing as in the proof of  \eqref{E:u2_final_est}, we obtain
\begin{equation}\label{E:u2.good_bound}
	{\|u_2(t)\|_{H^1}}+\|Ju_2(t)\|_{L^2} 
	\lesssim \varepsilon_2 t^{C_4 \varepsilon_1^2}
\end{equation}
for $t\ge1$. Combining this with \eqref{E:u2bound0to1}, we obtain the desired estimate \eqref{upper_bound}.

%\begin{rem}(あとで消す注意)
%前のバージョンでは,
%\[
%	\int_1^t s^{-1 + \varepsilon} ds =\varepsilon^{-1} (t^{\varepsilon}-1)
%\]
%の評価において, 係数の$\varepsilon^{-1}$を出さないようにするため,
%\[
%	\int_1^t s^{-1 + \varepsilon} ds \le t^{\varepsilon}\log t
%\]
%と置き換えて評価していた. しかし, 実はここで$\varepsilon_1^{-1}$が消えたように見えるのは見かけだけで, この後で,
%\[
%	\log t \le C_\varepsilon t^{\varepsilon}	
%\]
%などと$t$の小さい冪でさらに上から抑える際に, 定数として
%$C_\varepsilon = e^{-1} \varepsilon^{-1}$
%が出てしまう.
%この定数の評価も含めると$\delta$や$\varepsilon_1$を絞っただけでは矛盾を導くために必要な係数のsmallnessがでないことがわかった（やり方が悪かっただけの可能性はあるが）.
%
%そこで,この評価はやめて, Gronwall を正確に計算することにした.
%上の理由から$\varepsilon_1$を絞っただけでは必要なsmallnessがでないので, 
%$t$の範囲を大きいところに制限することで $t$の冪の差から所望の smallness を導出した.
%それで, $t=1$で評価の仕方を分けていたところを,
%大きい数$T_0$ でわけることになった.
%なお, 時刻$T_0$までは初期値問題を解いて bound を出している.
%また$T^\ast$の定義における係数$10$は, そうとる理由が乏しくなったため $4$に変更した.
%4も1の次（初期値問題として評価する時間の終了時刻$T_0$での解の大きさの係数）の次（最終的に$t=\infty$付近で得られるbound の係数）の次の整数という程度の意味しかないが.
%\end{rem}

\subsection{Asymptotic behavior}

We complete the proof of Theorem 1.1 by establishing (\ref{asymp2}).

%%%%%%%%%%%%%%%%%%%%%%%%
\begin{proof}[Proof of Theorem \ref{1.1}]
%%%%%%%%%%%%%%%%%%%%%%%%
Let %$\alpha \in L^\infty$, 
$\Phi_1 \in L^\infty$ and
$W_1\in L^\infty$ be the functions defined in the proof of \eqref{asymp1}.

If $W_1(\xi_0)=0$ at some $\xi_0 \in \R$
then
one sees from Lemmas \ref{3.1} and \ref{3.2}, \eqref{eq.w}, and \eqref{E:u2.good_bound} that
there exists $\beta_0=\beta_0(\xi_0)$ such that
\begin{equation}\label{asympa0}
	w(t,\xi_0) = \beta(\xi_0) + O(t^{-1/4 + C_4 \varepsilon_1^2}).
\end{equation}
Note that $\beta_0$ and the second term of the right hand side are bounded uniformly in $\xi_0 \in \R$.

We now consider the case $W_1(\xi_0) \neq 0$.
Plugging Lemmas \ref{3.1} and \ref{3.2} and \eqref{E:u2.good_bound} to (\ref{relation.w}), 
we see that there exist some $\beta_1 , \beta_2 \in L^{\infty}$ such that  
\begin{equation}\label{asymppf1}
 Q(t)^{-1}P^{-1} \begin{pmatrix}
w(\xi_0) \\
\overline{w(\xi_0)}
\end{pmatrix}
 = \begin{pmatrix}
\beta_1(\xi_0) \\
\beta_2(\xi_0) 
\end{pmatrix} 
+ O(\varepsilon_1^2\varepsilon_2
t^{-1/4 + C_1 \varepsilon^2_1 + 2\mu C_2 \varepsilon^2_1 + C_4 \varepsilon_1^2})
\end{equation}
%in $L^{\infty}$ 
as $t \to \infty$. 
The second term of the right hand side is bounded uniformly in $\xi_0 \in \R$.
Let us introduce
\[
	Z := \left\{ \begin{pmatrix}
z \\
\overline{z}
\end{pmatrix}
\in \mathbb{C}^2 \ \middle| \ z \in \mathbb{C}  \right\},
\]
which is a closed subspace of $\mathbb{C}^2$.
It follows from \eqref{asymppf1} that
\begin{align*}
	\lim_{t\to\infty} t^{-\mu |W_1|^2 } 
	P^{-1} \begin{pmatrix}
w(\xi_0) \\
\overline{w(\xi_0)}
\end{pmatrix} 
&{}= \lim_{t\to\infty} \begin{pmatrix} t^{-2\mu |W_1|^2} & 0 \\ 0 & 1 \end{pmatrix}  Q(t)^{-1}P^{-1} \begin{pmatrix}
w(\xi_0) \\
\overline{w(\xi_0)}
\end{pmatrix}
\\
&{}=\begin{pmatrix} 0 & 0 \\ 0 & 1 \end{pmatrix} \begin{pmatrix}
\beta_1(\xi_0) \\
\beta_2(\xi_0) 
\end{pmatrix}.
\end{align*}
This implies that
\begin{align*}
	P \begin{pmatrix} 0 & 0 \\ 0 & 1 \end{pmatrix} \begin{pmatrix}
\beta_1(\xi_0) \\
\beta_2(\xi_0) 
\end{pmatrix}
	&{}= P \lim_{t\to\infty} t^{-\mu |W_1(\xi_0)|^2 } 
	P^{-1} \begin{pmatrix}
w(\xi_0) \\
\overline{w(\xi_0)}
\end{pmatrix}\\
 &{}=\lim_{t\to\infty} t^{-\mu |W_1(\xi_0)|^2 } 
	\begin{pmatrix}
w(\xi_0) \\
\overline{w(\xi_0)}
\end{pmatrix} \in Z,
\end{align*}
since $Z$ is closed.
Further, if $\varepsilon_1$ is sufficiently small then one has
\begin{align*}
	&\lim_{t\to\infty} t^{\mu |W_1(\xi_0)|^2 } \left(
	P^{-1} \begin{pmatrix}
w(\xi_0) \\
\overline{w(\xi_0)}
\end{pmatrix}
- t^{\mu |W_1(\xi_0)|^2}
\begin{pmatrix} 0 & 0 \\ 0 & 1 \end{pmatrix} \begin{pmatrix}
\beta_1(\xi_0) \\
\beta_2(\xi_0) 
\end{pmatrix}
\right)\\
&{}=\lim_{t\to\infty}  
\begin{pmatrix} 1 & 0 \\ 0 & 0 \end{pmatrix}
	Q(t)^{-1}
	P^{-1} \begin{pmatrix}
w(\xi_0) \\
\overline{w(\xi_0)}
\end{pmatrix}\\
&{}\quad + \lim_{t\to\infty}  
\begin{pmatrix} 0 & 0 \\ 0 & t^{2\mu |W_1(\xi_0)|^2} \end{pmatrix}
\left(
Q(t)^{-1}
	P^{-1} \begin{pmatrix}
w(\xi_0) \\
\overline{w(\xi_0)}
\end{pmatrix}-
 \begin{pmatrix}
\beta_1(\xi_0) \\
\beta_2(\xi_0) 
\end{pmatrix}
\right)\\
&{}=\begin{pmatrix} 1 & 0 \\ 0 & 0 \end{pmatrix} \begin{pmatrix}
\beta_1(\xi_0) \\
\beta_2(\xi_0) 
\end{pmatrix}
\end{align*}
by means of \eqref{asymppf1}.
Hence
one sees that
\[
	P\begin{pmatrix} 1 & 0 \\ 0 & 0 \end{pmatrix} \begin{pmatrix}
\beta_1 \\
\beta_2 
\end{pmatrix}
\]
takes value in $Z$ by arguing as above. Hence,
\[
	P \begin{pmatrix}
\beta_1 \\
\beta_2 
\end{pmatrix}
	= P\begin{pmatrix} 1 & 0 \\ 0 & 0 \end{pmatrix} \begin{pmatrix}
\beta_1 \\
\beta_2 
\end{pmatrix} + P\begin{pmatrix} 0 & 0 \\ 0 & 1 \end{pmatrix} \begin{pmatrix}
\beta_1 \\
\beta_2 
\end{pmatrix}
\]
also takes value in $Z$. 
%We remark that $\| P \| \sim_{\lambda_3,\lambda_6} 1$ uniformly in $\alpha \neq 0$.
Thus, there exists $W_2 \in L^\infty(\{W_1\neq0\})$ such that
\[
	\begin{pmatrix}
\beta_1 \\
\beta_2 
\end{pmatrix}
	= P^{-1} 	\begin{pmatrix}
  W_2 \\
\overline{ W_2}
\end{pmatrix}.
\]
Then it follows that 
\begin{equation} \label{asymp.w.in.L^infty}
w 
 = \begin{pmatrix} 1 & 0  \end{pmatrix} P Q(t) 
P^{-1} 	\begin{pmatrix}
W_2 \\
\overline{W_2}
\end{pmatrix} + O(t^{-1/4 + C_1 \varepsilon^2_1 + 2\mu C_2 \varepsilon^2_1 + C_4 \varepsilon_1^2})
\end{equation}
in $L^\infty(\{W_1\neq0\})$.
We remark that,
in view of \eqref{asympa0},
the identity is valid also on the set $\{W_1=0\}$ by using the convention $P Q(t) P^{-1} = I_2$ and extending $W_2$ to the set by $W_2(\xi):=\beta_0(\xi)$.
Note that the extension $W_2$ belongs to $L^\infty(\mathbb{R})$.
%that is,
%\begin{equation}    \label{asymp.w.in.L^infty}
%\begin{aligned}
%w ={}& \lambda_6 e^{2i\theta} \beta_1 t^{-\mu |\alpha|^2} 
%       + (3\lambda_1 - 2 \lambda_6 - i\mu) \beta_2 t^{\mu |\alpha|^2} 
%       + O(t^{-1/4 + C_1 \varepsilon^2_1 + 2\mu C_2 \varepsilon^2_1 + \delta}).
%\end{aligned}
%\end{equation}
Let us now recall that $w = e^{3i\lambda_1 (\Phi_1-\theta_1)} v_2$. Then, in $L^\infty(\R)$, 
we have 
\begin{align*}
u_2 (t) 
%&= U(t) \mathcal{F}^{-1}v_2 (t) \\
&= M(t) D(t) \mathcal{F} M(t) \mathcal{F}^{-1}v_2(t) \\
&= M(t) D(t) v_2 (t) + O(t^{-3/4 + C_4 \varepsilon_1^2}) \\
&= M(t) D(t) e^{-3 i \lambda_1 (\Phi_1-\theta_1)} w +  O(t^{-3/4 + C_4 \varepsilon_1^2}).
\end{align*}
Plugging (\ref{asymp.w.in.L^infty}) and $e^{-3i \lambda_1( \Phi_1-\theta_1)} 
= e^{-3i \lambda_1 |W_1|^2 \log t } 
+ O(t^{-1/4 + C_1 \varepsilon^2_1})$ 
in $L^\infty$ to the identity,
we see that 
%\begin{align*}
%u_2(t) 
%&= M(t) D(t) e^{-3 i \lambda_1 (|\alpha|^2 \log t + \theta_1)} \\
%&\qquad \times \left[ \lambda_6 e^{2i\theta} \beta_1 t^{-\mu |\alpha|^2} 
%       + (3\lambda_1 - 2 \lambda_6 - i\mu) \beta_2 t^{\mu |\alpha|^2} \right] \\
%&\qquad \qquad + O(t^{-3/4 + C_1 \varepsilon^2_1 + 2\mu C_2 \varepsilon^2_1 + \delta}). 
%\end{align*}
\begin{align*}
u_2(t) 
&= M(t) D(t) e^{-3 i \lambda_1 |W_1|^2 \log t }  \begin{pmatrix} 1 & 0  \end{pmatrix} P Q(t) 
P^{-1} 	\begin{pmatrix}
W_2 \\
\overline{W_2}
\end{pmatrix}\\
%&\qquad \times \left[ \lambda_6 e^{2i\theta} \beta_1 t^{-\mu |\alpha|^2} 
%       + (3\lambda_1 - 2 \lambda_6 - i\mu) \beta_2 t^{\mu |\alpha|^2} \right] \\
&\qquad \qquad + O(t^{-3/4 + C_1 \varepsilon^2_1 + 2\mu C_2 \varepsilon^2_1 + C_4 \varepsilon_1^2}).
\end{align*}
%Write $e^{-3i\lambda_1 \theta_1} e^{2i \theta} \beta_1 = W_{2,1}$, 
%$e^{-3i\lambda_1 \theta_1} \beta_2 = W_{2,2}$ and 
%Noting that $|\alpha| = |W_1|$ and $e^{i\theta-3i\lambda_1 \theta_1}=W_1/|W_1|$ if $|\alpha|\neq0$, o
One sees that the leading part is written as in \eqref{asymp2}. 
This completes the proof of Theorem \ref{1.1}.
%Then we obtain (\ref{asymp2}). 
\end{proof}

%%%%%%%%%%%%%%%%%%%%%%%%%%%%%%%
%%%%%%%%%%%%%%%%%%%%%%%%%%%%%%%
\section{Proof of Theorem \ref{1.2}}
%%%%%%%%%%%%%%%%%%%%%%%%%%%%%%%
%%%%%%%%%%%%%%%%%%%%%%%%%%%%%%%
%Let ${{\mathcal L}}= i \partial_t+(1/2)\partial^2_x$. 
As mentioned in the introduction, we will 
construct the solution by solving
the final state problem:
\begin{equation}\tag{\ref{NLS110}}
\left\{ 
\begin{aligned}
   & {{\mathcal L}}u_1 
       = 3 \lambda_1 |u_1|^2 u_1, \quad && t \in \re, x \in \re,\\
   & {{\mathcal L}}u_2 
       = \lambda_6 (2|u_1|^2 u_2 + u^2_1 \overline{u}_2), \quad && t \in \re, x \in \re,\\
    &\| u_j(t)-u_{\mathrm{ap},j}(t)\|_{L^\infty} = o(\|u_{\mathrm{ap},j}(t)\|_{L^\infty}),  && j=1,2.
\end{aligned}
\right. 
\end{equation}

\subsection{Reformulation as an integral equation}
Let us first reformulate (\ref{NLS110}) as an integral equation employing the argument in \cite{HN3}. 
Introduce a modified approximate solution $\widetilde{u}_{\mathrm{ap},j}(t)$ defined by
\begin{equation*}
\widetilde{u}_{\mathrm{ap},j}(t):=
U(t){{\mathcal F}}^{-1}[F_{j}(t,\xi)]
\end{equation*}
with $F_{j}$ given by (\ref{F11}) and (\ref{F12}).
Let $v_{j}:=u_{j}-u_{\mathrm{ap},j}$ $(j=1,2)$ and let 
\begin{align*}
\tilde{{\mathcal N}}_{1}(u_{1},u_{2})&= 3\lambda_1 |u_1|^2 u_1,\\
\tilde{{\mathcal N}}_{2}(u_{1},u_{2})&= \lambda_6( 2|u_1|^2 u_2 + u^2_1 \overline{u}_2). 
\end{align*}
Then, at least formally, we have for $j=1,2$,
\begin{align}
{{\mathcal L}}(u_j - \widetilde{u}_{\mathrm{ap},j})
%&=  {{\mathcal L}}u_{j}-{{\mathcal L}}\widetilde{u}_{\mathrm{ap},j}
%\nonumber\\
&= \tilde{{\mathcal N}}_{j}(u_{1},u_{2})-{{\mathcal L}}\widetilde{u}_{\mathrm{ap},j}
\nonumber\\
&= \tilde{{\mathcal N}}_{j}(v_{1}+u_{\mathrm{ap},1},v_{2}+u_{\mathrm{ap},2})-{{\mathcal L}}\widetilde{u}_{\mathrm{ap},j}
\nonumber\\
&= \tilde{{\mathcal N}}_{j}(v_{1}+u_{\mathrm{ap},1},v_{2}+u_{\mathrm{ap},2})
-\tilde{{\mathcal N}}_{j}(u_{\mathrm{ap},1},u_{\mathrm{ap},2})+\mathcal{E}_{j},
\label{diff}
\end{align}
where
\begin{equation}\label{E:Ej}
\mathcal{E}_{j}:=-{{\mathcal L}}\widetilde{u}_{\mathrm{ap},j}
+\tilde{{\mathcal N}}_{j}(u_{\mathrm{ap},1},u_{\mathrm{ap},2}).
\end{equation}
By the Duhamel principle, (\ref{diff}) is written as
\begin{equation}
\label{INT}
\begin{aligned}
v_{j}(t)
&= i\int_{t}^{\infty}U(t-\tau)
\left\{\tilde{{\mathcal N}}_{j}(v_{1}+u_{\mathrm{ap},1},v_{2}+u_{\mathrm{ap},2})
-\tilde{{\mathcal N}}_{j}(u_{\mathrm{ap},1},u_{\mathrm{ap},2})\right\}(\tau)d\tau\\
&\quad +i\int_{t}^{\infty}U(t-\tau)\mathcal{E}_{j}(\tau)d\tau + \mathcal{R}_j.
\end{aligned}
\end{equation}
where
\begin{equation}\label{E:Rj}
	\mathcal{R}_j := \widetilde{u}_{\mathrm{ap},j}-u_{\mathrm{ap},j}.
\end{equation}
Note that the formula is chosen so that
the integrands decay as $t\to\infty$.

To show the existence of $(v_{1},v_{2})$ satisfying (\ref{INT}), we 
shall prove that the map $\Phi=(\Phi_{1}, \Phi_{2})$ defined by 
\begin{equation}
\begin{aligned}
&\Phi_{j}[(v_{1},v_{2})](t)\\
&= i\int_{t}^{\infty}U(t-\tau)
\left\{\tilde{{\mathcal N}}_{j}(v_{1}+u_{\mathrm{ap},1},v_{2}+u_{\mathrm{ap},2})
-{\tilde{\mathcal N}}_{j}(u_{\mathrm{ap},1},u_{\mathrm{ap},2})\right\}(\tau)d\tau
\nonumber\\
&\quad +i\int_{t}^{\infty}U(t-\tau)\mathcal{E}_{j}(\tau)d\tau + \mathcal{R}_j
%\label{INT1}
\end{aligned}
\end{equation}
is a contraction on
\begin{align*}
{{{\bf X}}}_{T}&=
\{(v_{1},v_{2})\in C([T,\infty);L^2(\R);\|(v_{1},v_{2})\|_{{{\bf X}_T}}\le 1\},\\
\|(v_{1},v_{2})\|_{{{\bf X}_T}}
&=\sup_{t\ge T}
(t^{\tilde{\nu}+\frac12}\|v_{1}\|_{L_x^{2}}+t^{\tilde{\nu}}\|Jv_{1}\|_{L_x^{2}}
+t^{\tilde{\nu}+\frac12-\delta}\|v_{2}\|_{L_x^{2}}
+t^{\tilde{\nu}-\delta}\|Jv_{2}\|_{L_x^{2}})
\end{align*}
for some $T\ge2$,
where $\tilde{\nu}:=\nu-1/2\in (0,1/2)$ and $\delta\in(0,\tilde{\nu})$ are
arbitrarily small numbers.

\begin{rem}
As seen below, the space ${\bf X}_T$ is chosen so that decay of $\|v_j\|_{L^\infty}$ can be deduced. 
\end{rem}

\begin{rem}
Our definition of a \emph{solution} to \eqref{NLS110} is a pair of functions $(v_1+u_{\mathrm{ap},1},v_2+u_{\mathrm{ap},2})$ such that $(v_1,v_2)\in {\bf X}_T$ holds for some $T\ge 2$ and that
$(v_1,v_2)$ satisfies \eqref{INT} in $C([T,\infty),L^2)$ sense.
Under a suitable regularity assumption on $u_{\mathrm{ap},j}$, one may see that $u_j:=v_j+u_{\mathrm{ap},j}$
belongs to $C([T,\infty),H^1)$. Then, $(u_1,u_2)$ satisfies the differential equation \eqref{NLS110} on $[T,\infty)$
in the $H^{-1}$ sense.
\end{rem}

\subsection{Estimates on approximate solutions}

Let us summarize decay properties 
of the given asymptotic profiles $u_{\mathrm{ap},j}$ and $\widetilde{u}_{\mathrm{ap},j}$
and of the differences $\mathcal{R}_j$ of those two. 
We also give an estimate of the error term $\mathcal{E}_j$.
We assume $\|W_1\|_{H_{\xi}^{2}}\le\varepsilon_{1}\le1$ and 
$\|W_2\|_{H_{\xi}^{2}}\le\varepsilon_{2}$.
\begin{lem}
Let $ u_{\mathrm{ap},1} $ be given by \eqref{E:uapj} and let $ \mathcal{R}_1 $ be defined by \eqref{E:Rj}.
It holds that
\begin{align}
\|u_{\mathrm{ap},1}\|_{L_x^{\infty}}
&\lesssim \varepsilon_{1}t^{-\frac12},\label{E:uap1-1}\\
%\|u_{\mathrm{ap},1}\|_{L_x^2}
%&\le&\|\alpha\|_{L_{\xi}^2},\\
\|\mathcal{R}_1\|_{L_x^2}
&\lesssim \varepsilon_{1}t^{-1}(\log t)^{2},
\label{E:uap1-2}\\
\|\mathcal{R}_1\|_{L_x^{\infty}}
&\lesssim \varepsilon_{1}t^{-\frac54}(\log t)^{2},%\|{{\mathcal F}}^{-1}F_1\|_{H_x^{0,2}},
\label{E:uap1-3}\\
\|Ju_{\mathrm{ap},1}\|_{L_x^2}
&\lesssim \varepsilon_{1}\log t
%\|x{{\mathcal F}}^{-1}F_1\|_{L_x^2}
%\le\log t\|\alpha\|_{H_{\xi}^1}
%(1+\|\alpha\|_{H_{\xi}^1}^2)
\label{E:uap1-4},\\
\|J\mathcal{R}_1\|_{L_x^2}
&\lesssim \varepsilon_{1} t^{-\frac12}(\log t)^{2},%\|{{\mathcal F}}^{-1}F_1\|_{H_x^{0,2}}.
\label{E:uap1-5}
\end{align}
for any $t\ge2$. 
In particular, 
\begin{equation}\label{E:uap1-6}
\|\widetilde{u}_{\mathrm{ap},1}\|_{L_x^{\infty}}\lesssim\varepsilon_{1}t^{-\frac12}
\end{equation}
for $t\ge2$.
\end{lem}
\begin{proof}
By definition, one immediately obtains \eqref{E:uap1-1}.
One has
\[
	\mathcal{R}_1
	=M(t) D(t) (\mathcal{F} M(t)\mathcal{F}^{-1}-1) F_1(t).
\]
Taking $L^p$ norm, we obtain
\begin{align*}
	\|\mathcal{R}_1\|_{L^p}
	&{}= t^{-\frac12+\frac1p}\|(\mathcal{F} M(t)\mathcal{F}^{-1}-1) F_1(t)\|_{L^p}.
\end{align*}	
When $p=2$, we have
\begin{align*}
	\|(\mathcal{F} M(t)\mathcal{F}^{-1}-1) F_1(t)\|_{L^2}&{}\lesssim t^{-1} \|\pt_x^2 F_1\|_{L^2}\\
	&{}\lesssim t^{-1} \|W_1\|_{H^2} \<\|W_1\|_{H^2}\>^4 (\log t)^2
\end{align*}
for $t\ge2$. Thus, \eqref{E:uap1-2} follows.
When $p=\infty$, mimicking the proof of \eqref{E:U-1}, one has
\begin{align*}
	\| (\mathcal{F} M \mathcal{F}^{-1} - 1) F_1(t)\|_{L^\infty}
	&{}\lesssim \| (\mathcal{F} M \mathcal{F}^{-1} - 1) F_1(t) \|_{L^2}^\frac12
	\|  (\mathcal{F} M \mathcal{F}^{-1} - 1) F_1(t) \|_{\dot{H}^1}^\frac12 \\
	&{}\lesssim (t^{-1} \|\pt_x^2 F_1\|_{L^2})^\frac12 (t^{-\frac12} \||\pt_x| F_1\|_{\dot{H}^1})^\frac12 \\
	&{}\lesssim t^{-\frac34} \|\pt_x^2 F_1\|_{L^2}\\
	&{}\lesssim t^{-\frac34} \|W_1\|_{H^2} \<\|W_1\|_{H^2}\>^4 (\log t)^2.
\end{align*}
We obtain \eqref{E:uap1-3}. 
Since $J(t) = M(t) (it\pt_x) M(-t)$, we have
\[
	Ju_{\mathrm{ap},1} = M(t) (it\pt_x) D(t) F_1(t) = M(t)D(t) i \pt_x F_1(t).
\]
Then, \eqref{E:uap1-4} immediately follows. Further, arguing as in the proof of \eqref{E:uap1-2},
one obtains \eqref{E:uap1-5}.
Finally, \eqref{E:uap1-6} is a consequence of \eqref{E:uap1-1} and \eqref{E:uap1-3}.
\end{proof}

\begin{lem}
Let $ u_{\mathrm{ap},2} $ be given by \eqref{E:uapj} and let $ \mathcal{R}_2 $ be defined by \eqref{E:Rj}.
It holds that
\begin{align}
\|u_{\mathrm{ap},2}\|_{L_x^{\infty}}
&\lesssim \varepsilon_{2}t^{-\frac12+C\varepsilon_{1}^{2}},\label{E:uap2-1}\\
\|\mathcal{R}_2\|_{L_x^2}
&\lesssim \varepsilon_{2}t^{-1+C\varepsilon_{1}^{2}}(\log t)^{2},
\label{E:uap2-2}\\
\|\mathcal{R}_2\|_{L_x^{\infty}}
&\lesssim \varepsilon_{2}t^{-\frac54+C\varepsilon_{1}^{2}}(\log t)^{2}, \label{E:uap2-3}\\
\|Ju_{\mathrm{ap},2}\|_{L_x^2}
&\lesssim \varepsilon_{2}t^{C\varepsilon_{1}^{2}}\log t,\label{E:uap2-4}\\
\|J\mathcal{R}_2\|_{L_x^2}
&\lesssim \varepsilon_{2} t^{-\frac12+C\varepsilon_{1}^{2}}(\log t)^{2} \label{E:uap2-5}
\end{align}
for $t\ge2$.
In particular, 
\begin{equation}\label{E:uap2-6}
\|\widetilde{u}_{\mathrm{ap},2}\|_{L_x^{\infty}}
\lesssim\varepsilon_{2}t^{-\frac12+C\varepsilon_{1}^{2}}
\end{equation}
for $t\ge2$.
\end{lem}
\begin{proof}
The proof of the estimate is similar to that for the corresponds estimate in the previous lemma.
We only give estimates on $F_2(t)$.
One has
\[
	\|F_2(t)\|_{L^p} \lesssim \|W_2\|_{L^p} t^{\tilde{\nu} \|W_1\|_{L^\infty}^2} \lesssim \varepsilon_2 t^{C\varepsilon_1^2}
\]
for $t\ge2$. This yields \eqref{E:uap2-1} and \eqref{E:uap2-4}.
In order to estimate $\| F_2\|_{H^2}$, let us first check that $F_2 \in C^1(\R)$.
The continuity of the factor
\[
	F: =\tfrac{W_1(\xi)^2}{|W_1(\xi)|^2} 
              (   t^{\tilde{\nu} |W_1(\xi)|^2} - t^{-\tilde{\nu} |W_1(\xi)|^2} )
\]
 at the zero point of $W_1$ is only problematic.
The continuity of $F$ at such points follows from the identity
\[
	F = W_1^2(2\tilde{\nu} \log t + o(1))=O(W_1^2)
\]
as $|W_1| \to 0$.
The continuity of $\pt_x F$ at such points is also verified by %combining the identity and
\begin{align*}
	\pt_x F 
%	={}& \left(2W_1 \pt_x W_1 - \frac{2W_1^2\Re (\overline{W_1} \pt_x W_1)}{|W_1|^2}\right)(2\tilde{\nu} \log t + o(1))\\
%	&+ \frac{2W_1^2\Re (\overline{W_1} \pt_x W_1)}{|W_1|^2} (2+o(1))\tilde{\nu} \log t
	={}& 2W_1 \pt_x W_1 (2\tilde{\nu} \log t + o(1))
	+ \frac{2W_1^2\Re (\overline{W_1} \pt_x W_1)}{|W_1|^2} o(1)
	=O(W_1)
\end{align*}
as $|W_1| \to 0$. Integrability of $|F_2|^2+|\pt_x F_2|^2$ in $\{|W_1|\le 1\}$
also follows from these identities.
Similarly, we have $|\pt_x^2 F|\lesssim |\pt_x^2 W_1|^2 + |W_1||\pt_x^2 W_1|$ as $|W_1| \to 0$, which yields the integrability of $|\pt_x^2 F|^2$
in $\{|W_1|\le 1\}$, together with $\pt_x W_1\in H^1 \hookrightarrow L^\infty$.
Hence, we have \eqref{E:uap2-2}, \eqref{E:uap2-3}, and \eqref{E:uap2-5}.

Finally, \eqref{E:uap2-6} follows from \eqref{E:uap2-1} and \eqref{E:uap2-3}.
\end{proof}

Next we estimate the error term $\mathcal{E}_j$. 
\begin{lem}
Let $ \mathcal{E}_j $	be given by \eqref{E:Ej}. 
It holds that
\begin{equation}
\label{N5}
\|\mathcal{E}_{1}(t)\|_{L_{x}^{2}}
\lesssim\varepsilon_{1}^{3}t^{-2}(\log t)^{2},
\end{equation}
\begin{equation}
\label{N6}
\|J\mathcal{E}_{1}(t)\|_{L_{x}^{2}}
\lesssim \varepsilon_{1}^{3}t^{-\frac32}(\log t)^{2}.
\end{equation}
\begin{equation}
\|\mathcal{E}_{2}(t)\|_{L_{x}^{2}}
\lesssim \varepsilon_{1}^{2}\varepsilon_{2}t^{-2+C\varepsilon_{1}^{2}}(\log t)^{2},
\label{N7}
\end{equation}
and
\begin{equation}
\|J\mathcal{E}_{2}(t)\|_{L_{x}^{2}}
\lesssim \varepsilon_{1}^{2}\varepsilon_{2}t^{-\frac32+C\varepsilon_{1}^{2}}(\log t)^{2}
\label{N8}
\end{equation}
for $t\ge2$.
\end{lem}
\begin{proof}
We remark that $(F_1,F_2)$ is a solution to \eqref{ODEsys}.
Hence,
\begin{eqnarray*}
i\pt_{t}F_{j}=t^{-1}{{\mathcal N}}_{j}(F_{1},F_{2})
\end{eqnarray*}
for $j=1,2$. Thanks to this identity, we have
\begin{equation}
\label{r1}
\begin{aligned}
{{\mathcal L}}\widetilde{u}_{\mathrm{ap},j}
&= U(t){{\mathcal F}}^{-1}(i\pt_{t}F_{j})
= U(t){{\mathcal F}}^{-1}[t^{-1}\tilde{{\mathcal N}}_{j}(F_{1},F_{2})].
\end{aligned}
\end{equation}
%Furthermore,
%\begin{eqnarray}
%\lefteqn{{{\mathcal N}}_{j}(\widetilde{u}_{\mathrm{ap},1},\widetilde{u}_{\mathrm{ap},2})}\qquad
%\nonumber\\
%&=&{{\mathcal N}}_{j}(u_{\mathrm{ap},1},u_{\mathrm{ap},2})
%+\left\{{{\mathcal N}}_{j}(\widetilde{u}_{\mathrm{ap},1},\widetilde{u}_{\mathrm{ap},2})
%-{{\mathcal N}}_{j}(u_{\mathrm{ap},1},u_{\mathrm{ap},2})\right\}
%\nonumber\\
%&=&
%t^{-1}U(t){{\mathcal F}}^{-1}[{{\mathcal N}}_{j}(F_{1},F_{2})]
%\nonumber\\
%& &-\left\{t^{-1}U(t){{\mathcal F}}^{-1}[{{\mathcal N}}_{j}(F_{1},F_{2})]-{{\mathcal N}}_{j}(u_{\mathrm{ap},1},u_{\mathrm{ap},2})\right\}
%\nonumber\\
%& &+\left\{{{\mathcal N}}_{j}(\widetilde{u}_{\mathrm{ap},1},\widetilde{u}_{\mathrm{ap},2})
%-{{\mathcal N}}_{j}(u_{\mathrm{ap},1},u_{\mathrm{ap},2})\right\},
%\label{r2}
%\end{eqnarray}
>From (\ref{r1}) and $u_{\mathrm{ap},j}=M(t)D(t) F_j$,  we obtain
\begin{align*}
\mathcal{E}_{j}
&= -U(t){{\mathcal F}}^{-1}[t^{-1}{\tilde{\mathcal N}}_{j}(F_{1},F_{2})]+{\tilde{\mathcal N}}_{j}(M(t)D(t) F_1,M(t)D(t) F_2)\\
&= -t^{-1}M(t)D(t)\mathcal{F}M(t)\mathcal{F}^{-1}\tilde{\mathcal{N}}_j(F_1, F_2)+t^{-1}M(t)D(t)\tilde{\mathcal{N}}_j(F_1, F_2) \\
&= -t^{-1}M(t)D(t)\mathcal{F}(M(t)-1)\mathcal{F}^{-1}\tilde{\mathcal{N}}_j(F_1, F_2).
\end{align*}
Hence, mimicking the proof of \eqref{E:uap1-2}, one has
\[
\|\mathcal{E}_{1}(t)\|_{L_{x}^{2}}
\lesssim t^{-2}\|\mathcal{N}_1(F_1, F_2)\|_{\dot{H}^2_\xi}\lesssim\varepsilon_{1}^{3}t^{-2}(\log t)^{2},
\]
and
\[
\|J\mathcal{E}_{1}(t)\|_{L_{x}^{2}}
\lesssim
t^{-\frac32}\|\mathcal{N}_1(F_1, F_2)\|_{\dot{H}^2_\xi}\nonumber\lesssim \varepsilon_{1}^{3}t^{-\frac32}(\log t)^{2}.
\]
In a similar way, we obtain \eqref{N7} and \eqref{N8}.
%\begin{equation}
%\|R_{2}(t)\|_{L_{x}^{2}}
%\lesssim \varepsilon_{1}^{2}\varepsilon_{2}t^{-2+C\varepsilon_{1}^{2}}(\log t)^{2},
%\label{N7}
%\end{equation}
%\begin{equation}
%\|JR_{2}(t)\|_{L_{x}^{2}}
%\lesssim \varepsilon_{1}^{2}\varepsilon_{2}t^{-\frac32+C\varepsilon_{1}^{2}}(\log t)^{2}.
%\label{N8}
%\end{equation}
\end{proof}

\subsection{Completion of the proof}

\begin{proof}[Proof of Theorem \ref{1.2}]
Fix $\nu \in (1/2,1)$ and $\delta \in (0,\nu-1/2)$. Let $\tilde{\nu}:=\nu-1/2$.
Then, one has $0<\delta<\tilde{\nu}<1/2$.
We first show that $\Phi$ is a map onto ${{\bf X}}_{T}$ for suitable 
$T\ge2$ if $\varepsilon_1$ is sufficiently small. 

Pick $(v_1,v_2) \in {{\bf X}}_{T}$.
We easily see  that
\begin{align*}
&|{{\mathcal N}}_{1}(v_{1}+u_{\mathrm{ap},1},v_{2}+u_{\mathrm{ap},2})
-{{\mathcal N}}_1(u_{\mathrm{ap},1},u_{\mathrm{ap},2})|\\
&\quad \lesssim
(|v_{1}|^{2} %+|u_{\mathrm{ap},1}||v_{1}|^{2}
+|u_{\mathrm{ap},1}|^{2})|v_{1}|,\\
&|{{\mathcal N}}_{2}(v_{1}+u_{\mathrm{ap},1},v_{2}+u_{\mathrm{ap},2})
-{{\mathcal N}}_2(u_{\mathrm{ap},1},u_{\mathrm{ap},2})|\\
&\quad \lesssim
|v_{1}|^{2}|v_{2}|+|u_{\mathrm{ap},2}||v_{1}|^{2}
+|u_{\mathrm{ap},1}||v_{1}||v_{2}|
+|u_{\mathrm{ap},1}||u_{\mathrm{ap},2}||v_{1}|
 +|u_{\mathrm{ap},1}|^{2}|v_{2}|.
\end{align*}
Furthermore, noting that the nonlinear terms 
are gauge invariant, we have
\begin{align*}
&|J{{\mathcal N}}_1
(v_1+u_{\mathrm{ap},1},
v_2+u_{\mathrm{ap},2})
-J{{\mathcal N}}_1
(u_{\mathrm{ap},1},
u_{\mathrm{ap},2})|\\
&\quad \lesssim
|v_1|^2|Jv_1|
+|Ju_{\mathrm{ap},1}||v_1|^2
+|u_{\mathrm{ap},1}||v_1||Jv_1|\\
&\qquad +|u_{\mathrm{ap},1}|
|Ju_{\mathrm{ap},1}||v_1|
+|u_{\mathrm{ap},1}|^2|Jv_1|,\\
&|J{{\mathcal N}}_2
(v_1+u_{\mathrm{ap},1},
v_2+u_{\mathrm{ap},2})
-J{{\mathcal N}}_2
(u_{\mathrm{ap},1},
u_{\mathrm{ap},2})|\\
&\quad \lesssim
|v_1||Jv_1||v_2|+|v_1|^2|Jv_2|
+|Ju_{\mathrm{ap},2}||v_1|^2
+|u_{\mathrm{ap},2}||v_1||Jv_1|\\
&\qquad 
+|Ju_{\mathrm{ap},1}||v_1||v_2|
+|u_{\mathrm{ap},1}||Jv_1||v_2|
+|u_{\mathrm{ap},1}||v_1||Jv_2|\\
&\qquad 
+|Ju_{\mathrm{ap},1}|
|u_{\mathrm{ap},2}||v_1|
+|u_{\mathrm{ap},1}|
|Ju_{\mathrm{ap},2}||v_1|
+|u_{\mathrm{ap},1}|
|u_{\mathrm{ap},2}||Jv_1|\\
&\qquad 
+|u_{\mathrm{ap},1}|
|Ju_{\mathrm{ap},1}||v_2|
+|u_{\mathrm{ap},1}|^2|Jv_2|.
\end{align*}
By plugging $\|v_1\|_{L^\infty} \lesssim t^{-1/2} \|v_1\|_{L^2}^{1/2} \|Jv_1\|_{L^2}^{1/2}$
and the estimates on $u_{\mathrm{ap},j}$ to those inequalities, we obtain
\begin{equation}
\label{N1}
\begin{aligned}
&\|{{\mathcal N}}_1
(v_1+u_{\mathrm{ap},1},
v_2+u_{\mathrm{ap},2})
-{{\mathcal N}}_1
(u_{\mathrm{ap},1},
u_{\mathrm{ap},2})\|_{L_x^2}\\
&\quad \lesssim
\|v_1\|_{L_x^{\infty}}^2\|v_1\|_{L_x^2}
+
\|u_{\mathrm{ap},1}\|_{L_x^{\infty}}^2
\|v_1\|_{L_x^2}\\
&\quad \lesssim
t^{-1}\|v_1\|_{L_x^2}^2\|Jv_1\|_{L_x^2}
+\|u_{\mathrm{ap},1}\|_{L_x^{\infty}}^2
\|v_1\|_{L_x^2}\\
&\quad \lesssim
t^{-3\tilde{\nu}-2}
+t^{-\tilde{\nu}-\frac32}\varepsilon_{1}^{2},
\end{aligned}
\end{equation}
\begin{equation}
\label{N2}
\begin{aligned}
&\|J{{\mathcal N}}_1
(v_1+u_{\mathrm{ap},1},
v_2+u_{\mathrm{ap},2})
-J{{\mathcal N}}_1
(u_{\mathrm{ap},1},
u_{\mathrm{ap},2})\|_{L_x^2}\\
&\quad \lesssim
\|v_1\|_{L_x^{\infty}}^2\|Jv_1\|_{L_x^{2}}
+\|Ju_{\mathrm{ap},1}\|_{L_x^2}
\|v_1\|_{L_x^{\infty}}^2
+\|u_{\mathrm{ap},1}\|_{L_x^{\infty}}
\|v_1\|_{L_x^{\infty}}
\|Jv_1\|_{L_x^{2}}\\
&\qquad 
+\|u_{\mathrm{ap},1}\|_{L_x^{\infty}}
\|Ju_{\mathrm{ap},1}\|_{L_x^{2}}
\|v_1\|_{L_x^{\infty}}
+\|u_{\mathrm{ap},1}\|_{L_x^{\infty}}^2
\|Jv_1\|_{L_x^{2}}\\
&\quad \lesssim
t^{-1}\|v_1\|_{L_x^2}\|Jv_1\|_{L_x^{2}}^2
+t^{-1}\|Ju_{\mathrm{ap},1}\|_{L_{x}^{2}}
\|v_1\|_{L_x^2}\|Jv_1\|_{L_x^{2}}\\
&\qquad 
+t^{-\frac12}\|u_{\mathrm{ap},1}\|_{L_{x}^{\infty}}
\|v_1\|_{L_x^2}^{\frac12}\|Jv_1\|_{L_x^{2}}^{\frac32}
+t^{-\frac12}
\|u_{\mathrm{ap},1}\|_{L_x^{\infty}}
\|Ju_{\mathrm{ap},1}\|_{L_x^2}
\|v_1\|_{L_x^2}^{\frac12}
\|Jv_1\|_{L_x^{2}}^{\frac12}
\\
&\qquad +
\|u_{\mathrm{ap},1}\|_{L_x^{\infty}}^2
\|Jv_1\|_{L_x^{2}}\\
&\quad \lesssim
t^{-3\tilde{\nu}-\frac32}
+t^{-2\tilde{\nu}-\frac54}\varepsilon_{1}
+t^{-\tilde{\nu}-1}\varepsilon_{1}^{2},
\end{aligned}
\end{equation}

\begin{equation}
\label{N3}
\begin{aligned}
&\|{{\mathcal N}}_2
(v_1+{u}_{\mathrm{ap},1},
v_2+{u}_{\mathrm{ap},2})
-{{\mathcal N}}_2
(u_{\mathrm{ap},1},
u_{\mathrm{ap},2})\|_{L_x^2}\\
&\quad \lesssim
\|v_1\|_{L_x^{\infty}}^2\|v_2\|_{L_x^2}
+\|u_{\mathrm{ap},2}\|_{L_x^{\infty}}
\|v_1\|_{L_x^{\infty}}
\|v_1\|_{L_x^2}
+
\|u_{\mathrm{ap},1}\|_{L_x^{\infty}}
\|v_1\|_{L_x^{\infty}}
\|v_2\|_{L_x^2}
\\
& \qquad 
+\|u_{\mathrm{ap},1}\|_{L_x^{\infty}}
\|u_{\mathrm{ap},2}\|_{L_x^{\infty}}
\|v_1\|_{L_x^{2}}
+
\|u_{\mathrm{ap},1}\|_{L_x^{\infty}}^2
\|v_2\|_{L_x^2}\\
&\quad \lesssim
t^{-3\tilde{\nu}-2+\delta}
+t^{-2\tilde{\nu}-\frac74+\delta}\varepsilon_{1}
+t^{-2\tilde{\nu}-\frac74+C\varepsilon_{1}^2}\varepsilon_{2}
+t^{-\tilde{\nu}-\frac32+\delta}\varepsilon_{1}^{2}
+t^{-\tilde{\nu}-\frac32+C\varepsilon_{1}^2}\varepsilon_{1}\varepsilon_{2},
\end{aligned}
\end{equation}
and
\begin{equation}
\label{N4}
\begin{aligned}
&\|J{{\mathcal N}}_2
(v_1+u_{\mathrm{ap},1},
v_2+u_{\mathrm{ap},2})
-J{{\mathcal N}}_2
(u_{\mathrm{ap},1},
u_{\mathrm{ap},2})\|_{L_x^2}\\
&\quad \lesssim
\|v_1\|_{L_x^{\infty}}
\|Jv_1\|_{L_x^2}
\|v_2\|_{L_x^{\infty}}
+\|v_1\|_{L_x^{\infty}}^2
\|Jv_2\|_{L_x^2}
+\|Ju_{\mathrm{ap},2}\|_{L_x^2}
\|v_1\|_{L_x^{\infty}}^2\\
& \qquad +\|u_{\mathrm{ap},2}\|_{L_x^{\infty}}
\|v_1\|_{L_x^{\infty}}
\|Jv_1\|_{L_x^2}
+\|Ju_{\mathrm{ap},1}\|_{L_x^2}
\|v_1\|_{L_x^{\infty}}
\|v_2\|_{L_x^{\infty}}\\
& \qquad +\|u_{\mathrm{ap},1}\|_{L_x^{\infty}}
\|Jv_1\|_{L_x^2}\|v_2\|_{L_x^{\infty}}
+\|u_{\mathrm{ap},1}\|_{L_x^{\infty}}
\|v_1\|_{L_x^{\infty}}
\|Jv_2\|_{L_x^2}\\
& \qquad +\|Ju_{\mathrm{ap},1}\|_{L_x^2}
\|u_{\mathrm{ap},2}\|_{L_x^{\infty}}
\|v_1\|_{L_x^{\infty}}
+\|u_{\mathrm{ap},1}\|_{L_x^{\infty}}
\|Ju_{\mathrm{ap},2}\|_{L_x^2}
\|v_1\|_{L_x^{\infty}}\\
& \qquad 
+\|u_{\mathrm{ap},1}\|_{L_x^{\infty}}
\|u_{\mathrm{ap},2}\|_{L_x^{\infty}}
\|Jv_1\|_{L_x^{2}}
+\|u_{\mathrm{ap},1}\|_{L_x^{\infty}}
\|Ju_{\mathrm{ap},1}\|_{L_x^2}
\|v_2\|_{L_x^{\infty}}\\
& \qquad 
+\|u_{\mathrm{ap},1}\|_{L_x^{\infty}}^2
\|Jv_2\|_{L_x^2}\\
&\quad \lesssim
t^{-3\tilde{\nu}-\frac32+\delta}
+t^{-2\tilde{\nu}-\frac54+\delta}\varepsilon_{1}
+t^{-2\tilde{\nu}-\frac54+C\varepsilon_{1}^2}\varepsilon_{2}
+t^{-\tilde{\nu}-1+\delta}\varepsilon_{1}^{2}\\
& \qquad 
+t^{-\tilde{\nu}-1+C\varepsilon_{1}^2}\varepsilon_{1}\varepsilon_{2}
\end{aligned}
\end{equation}
for $t\ge T$.
Applying (\ref{E:uap1-1})-(\ref{N4}) to the inequalities
\begin{align*}
&\|\Phi_j(v_{1},v_{2})(t)\|_{L^2}\\
&\quad \le \|\tilde{{\mathcal N}}_{1}(v_{1}+u_{\mathrm{ap},1},v_{2}+u_{\mathrm{ap},2})
-\tilde{{\mathcal N}}_j(u_{\mathrm{ap},1},u_{\mathrm{ap},2})\|_{L^1_tL^2_x(t,\infty)}\\
&\quad\quad + \|\mathcal{E}_j\|_{L^1_tL^2_x(t,\infty)} + \|\mathcal{R}_j(t)\|_{L^2}
\end{align*}
and
\begin{align*}
&\|J\Phi_j(v_{1},v_{2})(t)\|_{L^2}\\
&\quad \le \|J(\tilde{{\mathcal N}}_{1}(v_{1}+u_{\mathrm{ap},1},v_{2}+u_{\mathrm{ap},2})
-\tilde{{\mathcal N}}_j(u_{\mathrm{ap},1},u_{\mathrm{ap},2}))\|_{L^1_tL^2_x(t,\infty)}\\
&\quad\quad + \|J\mathcal{E}_j\|_{L^1_tL^2_x(t,\infty)} + \|J\mathcal{R}_j(t)\|_{L^2},
\end{align*}
we obtain
\begin{equation}\label{E:Phi_contraction_pf}
\begin{aligned}
\|\Phi(v_{1},v_{2})\|_{{{\bf X}}_T}%\\&\quad 
&\le \tilde{C}(T^{-2\tilde{\nu}-\frac12}
+(\varepsilon_{1}+\varepsilon_{2})T^{-\tilde{\nu}-\frac14}
+ \varepsilon_{1}\varepsilon_{2}T^{-\frac{\delta}2}\\
&\ \ \ \ \quad+(\varepsilon_{1}+\varepsilon_{2})T^{\tilde{\nu}-\frac12}(\log T)^2+\varepsilon_{1}^2 )
\end{aligned}
\end{equation}
%\begin{equation}\label{E:Phi_contraction_pf}
%\|\Phi(v_{1},v_{2})\|_{{{\bf X}}}
% \le \tilde{C}(\rho^{3}+\varepsilon_{1}\rho^{2}
%+\varepsilon_{1}(\varepsilon_{1}+\varepsilon_{2})\rho
%+\varepsilon_{1}^{2}(\varepsilon_{1}+\varepsilon_{2}))
%\end{equation}
if $C\varepsilon_1^2 \le \delta/2$, where $\tilde{C}=\tilde{C}(\tilde{\nu},\delta)$.
Pick $\tilde{\varepsilon}_1>0$ so small that 
$\tilde{C}\tilde{\varepsilon}_{1}^2 \le 1/2$ 
and $C\tilde{\varepsilon}_1^2 \le \delta/2$ are satisfied.
Then, there exists $T=T(\tilde{\varepsilon}_1(\tilde{\nu},\delta),\varepsilon_2,\tilde{\nu},\delta)=T(\nu,\delta,\varepsilon_2)>0$ such that
if $\varepsilon_{1}\le \tilde{\varepsilon}_1$ then 
the right hand side of \eqref{E:Phi_contraction_pf} is less than or equal to one. 
Then, $\Phi$ is a map onto ${{\bf X}}_{T}$. 
In a similar way, one concludes that $\Phi$ is a contraction 
map on ${{\bf X}}_{T}$ by letting $T$ even larger if necessary.
We leave the details to the reader. 
Thus, by Banach fixed point theorem 
one infers  that $\Phi$ has a unique fixed point in ${{\bf X}}_{T}$ 
which is the solution to the final state problem (\ref{NLS110}). 

Since $\|(v_1,v_2)\|_{{\bf X}_T} \le 1$, recalling that $\nu = \tilde{\nu} + 1/2$, we have
\[
	\|v_1(t)\|_{L^\infty} \lesssim t^{-\frac12} \|v_1(t)\|_{L^2}^\frac12 \|Jv_1(t)\|_{L^2}^\frac12
	\le  t^{-\nu-\frac14}
\]
and
\[
	\|v_2(t)\|_{L^\infty} \lesssim t^{-\frac12} \|v_2(t)\|_{L^2}^\frac12 \|Jv_2(t)\|_{L^2}^\frac12
	\le  t^{-\nu-\frac14+\delta}
\]
for $t\ge T$. This is \eqref{E:T2_error}. 
The estimate \eqref{E:T2_error2} immediately follows in a similar way. 
This completes the proof of Theorem \ref{1.2}. 
\end{proof}

\subsection*{Acknowledgments.} 
N.K. was supported by JSPS KAKENHI Grant Number JP17K05305. 
S.M. was supported by JSPS KAKENHI Grant Numbers JP17K14219, JP17H02854, JP18KK0386, JP21H00991, and JP21H00993.
J.S. was supported by JSPS KAKENHI Grant
Numbers %JP17H02851, 
JP19H05597, \\ 
JP20H00118, JP21H00993, and JP21K18588. 
K.U. was supported by JSPS KAKENHI Grant Number JP19K14578.

%\bigskip 
%
%\textsc{Faculty of Advanced Science and Technology, Kumamoto University, 
%Kumamoto, 860-8555, Japan}
%
%\textit{Email address}: \texttt{nkita@kumamoto-u.ac.jp} \\
%
%\textsc{Department of systems innovation, Graduate school of Engineering Science, Osaka University,
%Toyonaka Osaka, 560-8531, Japan}
%
%\textit{Email address}: \texttt{masaki@sigmath.es.osaka-u.ac.jp} \\
%
%\textsc{Faculty of Mathematics, Kyushu University, Fukuoka, 819-0395, Japan} 
%
%\textit{Email address}: \texttt{segata@math.kyushu-u.ac.jp} \\
%
%\textsc{Department of Applied Mathematics, Faculty of Science, Okayama University of Science, Okayama,
%700-0005, Japan}
%
%\textit{Email address}: \texttt{uriya@xmath.ous.ac.jp}

\end{document}